\documentclass[1 [leqno,11pt]{amsart}
\usepackage{amssymb, amsmath}
\allowdisplaybreaks[4]
\def\refeq#1{~(\ref{#1})}
\def\ccite#1{~\cite{#1}}

\def\inte#1{
\displaystyle\mathop{#1\kern0pt}^\circ }

\let\al=\alpha

\let\e=\varepsilon

\let\lam=\lambda
\let\r=\rho
\let\s=\sigma
\let\f=\frac

\let\p=\psi

\let\D=\Delta

\let\Om=\Omega
\let\wt=\widetilde

\def\cC{{\mathcal C}}

\def\cF{{\mathcal F}}

\def\cS{{\mathcal S}}

\def\la{\lambda}

\def\h{{\rm h}}

\def\virgp{\raise 2pt\hbox{,}}
\def\cdotpv{\raise 2pt\hbox{;}}

\def\eqdefa{\buildrel\hbox{\footnotesize def}\over =}

\def\C{\mathop{\mathbb C\kern 0pt}\nolimits}
\def\DD{\mathop{\mathbb D\kern 0pt}\nolimits}
\def\EE{\mathop{  {\mathbb E \kern 0pt}}\nolimits}
\def\K{\mathop{\mathbb K\kern 0pt}\nolimits}
\def\N{\mathop{\mathbb N\kern 0pt}\nolimits}
\def\Q{\mathop{\mathbb Q\kern 0pt}\nolimits}
\def\R{\mathop{\mathbb R\kern 0pt}\nolimits}
\def\SS{\mathop{\mathbb S\kern 0pt}\nolimits}
\def\ZZ{\mathop{\mathbb Z\kern 0pt}\nolimits}
\def\TT{\mathop{\mathbb T\kern 0pt}\nolimits}
\def\P{\mathop{\mathbb P\kern 0pt}\nolimits}
\newcommand{\ds}{\displaystyle}

\newcommand{\Z}{{\ZZ}}

\def\dive{\mathop{\rm div}\nolimits}

\def\Supp{\mathop{\rm Supp}\nolimits\ }

\def\na{\nabla}
\def\p{\partial}

\def\dH{\dot{H}}
\def\dB{\dot{B}}
\def\ga{\gamma}

\newcommand{\beq}{\begin{equation}}
\newcommand{\eeq}{\end{equation}}
\newcommand{\ben}{\begin{eqnarray}}
\newcommand{\een}{\end{eqnarray}}
\newcommand{\beno}{\begin{eqnarray*}}
\newcommand{\eeno}{\end{eqnarray*}}
\newcommand{\andf}{\quad\hbox{and}\quad}
\newcommand{\with}{\quad\hbox{with}\quad}
\newtheorem{defi}{Definition}[section]
\newtheorem{thm}{Theorem}[section]
\newtheorem{lem}{Lemma}[section]
\newtheorem{rmk}{Remark}[section]
\newtheorem{col}{Corollary}[section]
\newtheorem{prop}{Proposition}[section]

\begin{document}
\title[Fujita-Kato solution of 3-D inhomogeneous Navier-Stokes system]
{Global Fujita-Kato solution of   3-D inhomogeneous incompressible Navier-Stokes system
}
\author[P. Zhang]{Ping Zhang} \address[P. Zhang]{Academy of Mathematics $\&$ Systems Science
and  Hua Loo-Keng Key Laboratory of Mathematics, The Chinese Academy of
Sciences, Beijing 100190, CHINA, and School of Mathematical Sciences, University of Chinese Academy of Sciences, Beijing 100049, China.} \email{zp@amss.ac.cn}

\date{\today}

\begin{abstract} In this paper, we shall prove the global existence of weak solutions to 3D inhomogeneous incompressible Navier-Stokes system $({\rm INS})$
with initial density in the bounded function space and having a
positive lower bound and with initial velocity being sufficiently small  in the critical Besov space, $\dot B^{\f12}_{2,1}.$  This
result corresponds to the Fujita-Kato solutions of the classical Navier-Stokes system. The same idea can be used
to prove the global existence of weak solutions in the \emph{critical functional framework}  to $({\rm INS})$  with one component of the initial
velocity being large and can also be applied to provide a lower bound for the lifespan of smooth enough solutions
of $({\rm INS}).$
\end{abstract}
\maketitle



\setcounter{equation}{0}
\section{Introduction}

In this paper, we consider the  global existence of
weak solutions to the following three dimensional
incompressible inhomogeneous Navier-Stokes equations with initial
density in the bounded function space and having a positive lower bound,  and with initial velocity being sufficiently small in the critical Besov space, $\dB^{\f12}_{2,1},$
\begin{equation}\label{S1eq1}
\left\{
\begin{array}{ll}
\p_t\rho+u\cdot\na\rho=0,\qquad (t,x)\in\R^+\times\R^3,\\
\rho(\p_tu+u\cdot\na u)-\Delta u+\na \pi=0, \\
\textrm{div} u=0,\\
(\rho,u)|_{t=0}=(\rho_0,u_0),
\end{array}
\right.
\end{equation}
where $\rho, u$ stand for the density and  velocity of the fluid
respectively, $\pi$  is a scalar pressure function which guarantees the divergence free condition
of the velocity field.
 Such a system describes a fluid
which is obtained by mixing several immiscible fluids that are
incompressible and that have different densities.

 \smallbreak
 Let us first state  three major basic  features of system \eqref{S1eq1} in general Euclidean space $\R^d$, $d\geq 2$.
 Firstly,
 the incompressibility condition on  the convection velocity field in the density transport equation ensures  that
  \beq \label{S1eq13}
 \|\rho(t)\|_{L^\infty}
 =\|\rho_0\|_{L^\infty}
 \andf
 \mbox{meas}\bigl\{\ x\in\R^d\ |\ \al\leq\rho(t,x)\leq \beta\ \bigr\}\ \
 \mbox{is independent of}\ t\geq 0, \eeq
 for any pair of non-negative  real numbers $(\al,\beta)$.
 Secondly,   this system has the following energy law
   \beq \label{S1eq11} \frac 12 \int_{\R^d}\rho(t,x)|u(t,x)|^2 \textrm{d}x
 +\int_{0}^t\|\nabla u(t',\cdot)\|_{L^2(\R^d)}^2 \textrm{d}t'
 =\frac 12
 \int_{\R^d}\rho_0(x)|u_0(x)|^2 \textrm{d}x. \eeq
 The third  basic feature is
 the scaling invariance property: if~$(\rho,
 u,\pi)$ is a solution of \eqref{S1eq1} on~$[0,T]\times \R^d$,
 then the rescaled triplet~$(\rho, u,\pi)_\lam$ defined by
 \beq
 \label{S1eq12}
  (\rho, u
 ,\pi)_\lam(t,x) \eqdefa \bigl(\rho(\lam^2t, \lam x) , \lam u
 (\lam^2t, \lam x) ,\lam^2 \pi(\lam^2t,\lam x)\bigr),
 \quad \lambda\in\R\eeq is also a
 solution of \eqref{S1eq1}  on $[0,T/\la^2]\times \R^d$. This leads to the
 notion of critical regularity.

 \smallbreak
 Based on the energy estimate\refeq{S1eq11},
 Simon \ccite{Simon} (see \cite{LP} for  general result on the  variable viscosity case) constructed   global weak solutions of the system \eqref{S1eq1} with finite energy.
 Ladyzhenskaya and
  Solonnikov   \cite{LS} proved the local well-posedness to the system \eqref{S1eq1} with homogeneous Dirichlet
  boundary condition and with smooth initial data that has no vacuum. Motivated by \eqref{S1eq12}, Danchin \cite{danchin}
  established the well-posedness of \eqref{S1eq1}  in the so-called \emph{critical functional framework}
  for small perturbations of some positive constant density. The essential
  idea in  \cite{danchin} is to use functional spaces (or norms) that have the same
  \emph{scaling invariance} as \eqref{S1eq12}.
   In this framework, it has                                                     been stated in \cite{abidi, danchin} that for the initial data
  $(\r_0, u_0)$ satisfying \beno (\r_0-1)\in
  \dot{B}^{\f{d}p}_{p,1}(\R^d),\  u_0\in
  \dot{B}^{\f{d}p-1}_{p,1}(\R^d)\ \mbox{with}\ \dive u_0=0 \eeno and
  that for a small enough constant $c$ \beno
  \|\r_0-1\|_{\dot{B}^{\f{d}p}_{p,1}}+\|u_0\|_{\dot{B}^{\f{d}p-1}_{p,1}}\leq
  c, \eeno we have for any $p\in [1, 2d[$
  \begin{itemize}
  \item existence of global solution $(\r, u, \na p)$ with $\r-1\in
  C([0,\infty[;$ $\dot{B}^{\f{d}p}_{p,1}(\R^d)), $  $u\in
  C([0,\infty[;$ $\dot{B}^{\f{d}p-1}_{p,1}(\R^d)),$ and $\p_tu,
  \na^2 u, \na p\in L^1(\R^+;\dot{B}^{\f{d}p-1}_{p,1}(\R^d));$
  \item uniqueness in the above space if in addition $p\leq d.$
  \end{itemize}
  The above existence result
  was extended to general Besov spaces   in
  \cite{AP,dm, PZ2} even without the size restriction for the initial density \cite{AGZ2, AGZ3}. The uniqueness
  of such solutions for $p\in ]d,2d[$ was obtained by Danchin and Mucha in \cite{dm}.

 In all these aforementioned works, the density has to be at
 least in the Besov space $\dot B^{d/p}_{p,\infty}(\R^d)$, which excludes the density function with  discontinuities  across some hypersurface.
  Indeed,
   the Besov regularity of the characteristic function of a smooth domain is only $\dot B^{1/p}_{p,\infty}(\R^d).$
   Therefore, those results do not apply to a  mixture flow composed of two separate fluids with different densities.

 In particular, Lions proposed the following open question in
  \cite{LP}: suppose the initial density $\r_0={\bf 1}_{D}$ for some
  smooth domain $D,$ Theorem 2.1 of \cite{LP} provides at least
  one global weak solution $(\r,u)$ of \eqref{S1eq1} such that for all
  $t\geq 0,$ $\r(t)={\bf 1}_{D(t)}$ for some set $D(t)$ with ${\rm vol}(D(t))={\rm vol}(D).$
  Then whether or not the regularity of $D$ is preserved by  time
  evolution? To avoid the difficulty caused by vacuum, Liao and the author \cite{LZ1}
  investigated the case when the system \eqref{S1eq1} is supplemented with  the initial
 density, $\r_0(x)=\eta_1{\bf 1}_{\Om_0}+\eta_2{\bf 1}_{\Om_0^c},$ for a pair of positive   constants $(\eta_1,\eta_2)$ with $|\eta_1-\eta_2|$
 being sufficiently small,  and where $\Om_0$ is a
  bounded, simply connected  2D domain with $W^{k+2,p}$-boundary regularity for $k\in\N.$
  This smallness assumption for the difference between $\eta_1$ and $\eta_2$ was removed by the authors
    in \cite{LZ2}. Danchin and Zhang \cite{DZX}, Gancedo and Garcia-Juarez \cite{GG}  proved the propagation of $C^{k+\gamma}$ regularity
   of the density patch to \eqref{S1eq1}. Lately Danchin and Mucha \cite{dm3} can allow vacuum.

  In the general case when $\r_0\in L^\infty$ with a positive lower bound and  initial velocity $u_0\in H^1,$
  Kazhikov \cite{Ka74} proved the local existence of weak solution to the system \eqref{S1eq1}.
   While with $u_0\in H^2,$  Danchin and  Mucha \cite{dm2} proved that the  system \eqref{S1eq1} has a unique local in time solution.
 Paicu,  the  author and Zhang  \ccite {PZZ1}  improved the well-posedness results in \cite{dm2} with initial velocity in $H^s(\R^2)$
 for any $s>0,$ and with initial velocity in $H^1(\R^3).$   Chen, Zhang and Zhao \cite{CZZ16} further improved the regularity of the initial
 velocity in 3D to be in $H^s(\R^3)$  for any $s>\f12.$ Nevertheless, in either  \cite{CZZ16} or \cite{PZZ1}, the authors can not prove the propagation
 of the regularities for the initial velocity field, namely, they can not prove the velocity $u$ belongs to $C([0,T]; H^s).$ Furthermore,
 the  norms  of the initial velocity in \cite{CZZ16, PZZ1} is not critical in the sense that the norms are not scaling invariant
 under the transformation \eqref{S1eq12}.

On the other hand,  when~$\rho_0\equiv 1$, the
 system \eqref{S1eq1} reduces to the classical incompressible
 Navier-Stokes system (NS).  Let us recall the following celebrated result by Fujita and Kato \cite{KT64} on (NS):

 \begin{thm}  \label{KT64}
 {\sl   Given solenoidal vector field $u_0\in \dot H^{\frac 1 2}$  with $ \|u\|_{\dot H^{\frac 1 2}}\leq \e_0$ for $\e_0$ sufficiently
 small, then (NS) has a unique global solution $u\in C([0,\infty[; \dH^{\f12})\cap L^4(\R^+; \dH^1)\cap L^2(\R^+; \dH^{\f32}).$    }
 \end{thm}

Before proceeding, let us recall the definition of weak solutions to \eqref{S1eq1} from \cite{ HPZ2, PZZ1}:

\begin{defi}\label{defi1.1} {\sl We call $(\r,u, \na\pi)$ a global weak solution of \eqref{S1eq1}
if
\begin{itemize}
\item for any test function $\phi\in
C^\infty_c([0,\infty[\times\R^3),$ there holds
\beq\label{def1.1a}\begin{split}
\int_0^\infty\int_{\R^3}&\r(\p_t\phi+u\cdot\na\phi)\,dx\,dt+\int_{\R^3}\phi(0,x)\r_0(x)\,dx=0,\\
& \int_0^\infty\int_{\R^3}\dive u\phi\,dx\,dt=0,\end{split} \eeq

\item for any vector valued function $\Phi=(\Phi^1,\Phi^2,\Phi^3)\in
C_c^\infty([0,\infty[\times\R^3),$  one has \beq\label{def1.1b}
\int_0^\infty\int_{\R^3}\Bigl\{u\cdot\p_t\Phi-(u\cdot\na u) \cdot
\Phi -\f1\r\bigl(\D
u-\na\pi\bigr)\Phi\Bigr\}\,dx\,dt+\int_{\R^3}u_0\cdot\Phi(0,x)\,dx=0.
\eeq
\end{itemize}}
\end{defi}

The goal of the following theorem is to prove similar version of Theorem \ref{KT64} for the inhomogeneous incompressible Navier-Stokes
system \eqref{S1eq1},  the proof of which will be based on the basic features of \eqref{S1eq1}, namely, \eqref{S1eq13}, \eqref{S1eq11} and \eqref{S1eq12}.

\begin{thm}\label{thm1}
{\sl Let $(\rho_0,u_0)$ satisfy \beq\label{S1eq2}
0<c_0\le \rho_0(x)\le C_0<+\infty \andf u_0\in B^{\f12}. \eeq Then there
exists a constant $\varepsilon_0>0$ depending only on  $c_0, C_0$ such
that if \beq\label{S1eq3} \|u_0\|_{B^{\f12}}\le
\varepsilon_0, \eeq
the system (\ref{S1eq1}) has a  global weak solution  $(\r, u)$ with
$\rho\in L^\infty(\R^+\times\R^3)$ and $u\in C([0,\infty[, B^{\f12})\cap L^2(\R^+;B^{\f32})$ which satisfies
\beq \label{S1eq14} c_0\leq \rho(t,x) \leq C_0 \quad \mbox{for}\ (t,x)\in \R^+\times\R^3,
\eeq and
\beq\label{S1eq7}
\begin{split}
\|u\|_{\wt{L}^\infty_t(B^{\f12})}&+\|u\|_{\wt{L}^2_t(B^{\f32})}+\|\sqrt{t} u\|_{\wt{L}^\infty_t(\dB^{\f32})}
+\bigl\|\sqrt{t}(\na u, \pi)\bigr\|_{L^2_t(B^{\f12}_{6,1})}\\
&+\|\sqrt{t}u_t\|_{\wt{L}^2_t(B^{\f12})}+\|tu_t\|_{\wt{L}^\infty_t(B^{\f12})}+\|t D_t u\|_{\wt{L}^2_t(B^{\f32})}\leq C\|u_0\|_{B^{\f12}}.
\end{split}
\eeq Here and in all that follows, we always denote $D_t\eqdefa \p_t+u\cdot\na$ to be the material derivative.  For simplification, we always
denote $B^{s}\eqdefa \dB^s_{2,1}$ in this paper. The definitions of Besov spaces, $\dB^s_{p,r},$ and  Chemin-Lerner type space, $\widetilde{L}^{q}_T(B^s),$ will be recalled in the Appendix
\ref{apA}.
}
\end{thm}

\begin{rmk}  \begin{itemize}

\item[(1)] We improve the regularity of the initial velocity in $H^s$ for any $s>\f12$  in \cite{CZZ16} to be the critical space $B^{\f12}.$
Moreover, we can propagate the regularity of the initial velocity field, namely, here $u\in C([0,\infty[, B^{\f12}).$
Whereas the velocity field in  \cite{CZZ16} belongs to some time-weighted integer Sobolev spaces. More precisely, for any $t\in [0,+\infty[$ and
 $\s(t)\eqdefa\min(1,t),$ they have
\beno
\begin{split}
&\|\sigma(t)^{\f{1-s}2}\na u\|_{L^\infty_t(L^2)}+\bigl\|\sigma(t)^{\f{1-s}2}(\na^2u,u_t,\na\pi)\bigr\|_{L^2_t(L^2)}\le  C\|u_0\|_{H^s}^2,\\
&\bigl\|{\sigma(t)}^{1-\f{s}2}(\na^2u,u_t,\na\pi)\bigr\|_{L^\infty_t(L^2)}+\|{\sigma(t)}^{1-\f{s}2}\na u_t\|_{L^2_t(L^2)}
\le C\|u_0\|_{H^s}^2\exp\bigl(C(\|u_0\|_{L^2}^2+\|
u_0\|_{H^s}^4)\bigr).
\end{split} \eeno

 \item[(2)] The time weight in \eqref{S1eq7} is optimal even for heat semigroup $e^{t\D} u_0.$ Indeed, it follows from
 Lemma \ref{S4lem0} and Bernstein inequality that
 \beno
 \begin{split}
 \|\D_j(t^{\f12}\na e^{t\D}u_0)\|_{L^2}\lesssim& t^{\f12}2^j e^{-ct2^{2j}}\|\D_j u_0\|_{L^2}
 \lesssim  \|\D_j u_0\|_{L^2},
 \end{split}
 \eeno
 which implies that
 $  \| t^{\f12}\na e^{t\D}u_0 \|_{\wt{L}^\infty_t(B^s)}\lesssim \|u_0\|_{B^s}$ for any $s\in\R.$

 \item[(3)]  As in \cite{CZZ16, PZZ1}, with a little bit more regularity assumption on the initial velocity
 field, we can also prove the uniqueness of such weak solutions of \eqref{S1eq1} constructed in Theorem \ref{thm1}.

  \end{itemize}

\end{rmk}

 Again in  a \emph{critical functional framework}, Huang, Paicu and the  author  \cite{HPZ2}
 proved the global existence of weak solutions to \eqref{S1eq1}        provided that the initial data satisfy the nonlinear smallness
  condition:
  $$ \bigl(\|\r_0^{-1}-1\|_{L^\infty}+\|u_0^h\|_{\dot{B}^{-1+\frac{3}p}_{p, r}}\bigr)\exp\Bigl(
  C_r\|u_0^3\|_{\dot{B}^{-1+\frac{3}p}_{p,r}}^{2r}\Bigr)\leq \e_0
  $$ for some positive constants $\e_0, C_r$ and $1< p<3,$
  $1<r<\infty,$ where $u_0^h=(u_0^1,u_0^{2})$ and
  $u_0=(u_0^h,u_0^3).$ With a little bit more regularity assumption on
  the initial velocity, they \cite{HPZ2} also proved the uniqueness of
  such solutions.  Danchin
 and the  author extended this result to the half-space setting in \cite{DZ1}.  Nonetheless as in \cite{CZZ16, PZZ1},
  the authors there can not prove the propagation of the fractional derivative
  for the initial velocity field. Moreover, The result in \cite{HPZ2} does not work for the index $r=1$
  due to technical reason (the application of maximal regularity estimate for heat semi-group forbids the case for $r=1.)$
 The purpose of the next theorem is to solve the aforementioned  questions.

  Toward this, let us denote $a\eqdefa\f1{\r}-1.$ Then we can reformulate \eqref{S1eq1} as
\begin{equation}\label{S4eq1}
\left\{
\begin{array}{ll}
\p_ta+u\cdot\na a=0,\qquad (t,x)\in\R^+\times\R^3,\\
\p_tu+u\cdot\na u-(1+a)\bigl(\Delta u-\na \pi)=0, \\
\textrm{div} u=0,\\
(a,u)|_{t=0}=(a_0,u_0),
\end{array}
\right.
\end{equation}

The second result of this paper states as follows:

\begin{thm}\label{thm2}
{\sl Let $a_0\in
L^\infty$ and $ u_0=(u_0^\h,u_0^3)\in {B}^{\f12}.$ Then
there exists a positive constant $\e_0$ so that if \beq
\label{small1} \eta \eqdefa
\bigl(\|a_0\|_{L^\infty}+\|u_0^h\|_{{B}^{\f12}}\bigr)\exp\bigl(C
\|u_0^3\|_{{B}^{\f12}}^{2}\bigr)\leq
\e_0, \eeq \eqref{S4eq1} has a global weak solution $(a,u)$ in the
sense of Definition \ref{defi1.1}, which satisfies
$u=v+e^{t\D}u_0$ and
\beq\label{S1eq8}
      \begin{split}
      \|v\|_{\wt{L}^\infty_t(B^{\f12})}&+\|v\|_{\wt{L}^2_t(B^{\f32})}+\|\sqrt{t}v\|_{\wt{L}^\infty_t(B^{\f32})}
      \\
      &+\|\sqrt{t}v_t \|_{\wt{L}^2_t(B^{\f12})}+\|tv_t\|_{\wt{L}^\infty_t(B^{\f12})}+\|t D_tv\|_{\wt{L}^2_t(B^{\f32})}\leq C\eta.
      \end{split}
      \eeq
}
\end{thm}

The idea of the proof to Theorem \ref{thm1} can also be used to prove the
 following Theorem concerning the lifespan of smooth enough solution to \eqref{S1eq1}, which in particular
  generalize the corresponding result for classical Navier-Stokes system (see Proposition 1.1 of \cite{CG}) to the inhomogeneous context.

\begin{thm}\label{thm3}
{\sl Let $\r_0$ satisfy \eqref{S1eq2}, $u_0\in H^{\f12+2\ga}$ for $\ga\in ]0,1/4].$ Then \eqref{S1eq1} has a unique solution
$(\r, u)$ on $[0,T]$ so that $u\in C([0,T]; H^{\f12+2\ga})\cap L^2(]0,T[; \dH^{\f32+2\ga}).$
 We denote $T^\ast(u_0)$ to be the maximal time of existence of such solution.
Then there exists a constant $c_\ga$ so that
                              \beq \label{S1eq10}
T^\ast(u_0)\geq c_\ga \|u_0\|_{\dH^{\f12+2\ga}}^{-\f1\ga}\eqdefa T_\ga.
\eeq    Moreover, for $t\leq T_\ga,$ there holds
\beq \label{S1eq16}
\begin{split}
\|u\|_{\wt{L}^\infty_t(\dH^{\f12+2\ga})}&+\|u\|_{L^2_t(\dH^{\f32+2\ga})}+ \|\sqrt{t} u\|_{\wt{L}^\infty_t(\dH^{\f32+2\ga})}+\|\sqrt{t}u_t\|_{L^2_t(\dH^{\f12+2\ga})}\\
&\qquad+\|tD_tu\|_{\wt{L}^\infty_t(\dH^{\f12+2\ga})}+\|t D_tu\|_{\wt{L}^2_t(\dH^{\f32+2\ga})}\leq C\|u_0\|_{\dH^{\f12+2\ga}}.
\end{split}
\eeq
  }
\end{thm}

\begin{rmk}
When $\r_0=1$ and $\ga=\f14,$ \eqref{S1eq10} corresponds to the celebrated Leray  estimate
on the lifespan of strong solutions to the classical Navier-Stokes system in \cite{Leray}.
\end{rmk}

\medbreak  Let us end this introduction by some notations that will
be used in all that follows.

For operators $A,B,$ we denote $[A;B]=AB-BA$ to be commutator of $A$ and $B.$ For~$a\lesssim
b$, we mean that there is a uniform constant $C,$ which may be
different on different lines, such that $a\leq Cb$.  We denote by
$\int_{\R^3}f|g\,dx$ the $L^2(\R^3)$ inner product of $f$ and $g$. For $X$
a Banach space and $I$ an interval of $\R,$ we denote by $C(I;\,X)$
the set of continuous functions on~$I$ with values in $X.$    For
$q$ in~$[1,+\infty],$ the notation $L^q(I;\,X)$ stands for the set
of measurable functions on $I$ with values in $X,$ such that
$t\longmapsto\|f(t)\|_{X}$ belongs to $L^q(I).$ Finally we always denote $(d_j)_{j\in\Z}$ (resp. $(c_j)_{j\in\Z}$) to be a generic element of $\ell^1(\Z)$
(resp. $\ell^2(\Z)$) so that
$\sum_{j\in\Z}d_j=1$ (resp. $\sum_{j\in\Z}c_j^2=1$).

\medskip

 \setcounter{equation}{0}
\section{ Ideas of the proof and structure of the paper}\label{Sect2}

First of all, let us recall that the classical idea to prove the local wellposedness of \eqref{S4eq1} in the
{\it critical functional  framework} in \cite{abidi, danchin} is first to apply the operator $\D_j$ (see \eqref{S1eq5}) to the
momentum equation of \eqref{S4eq1} and then taking $L^2$ inner product of the result equation with $\D_ju,$ which gives
\beno
\f12\f{d}{dt}\|\D_j u(t)\|_{L^2}^2+\|\na\D_j u\|_{L^2}^2=\bigl(\D_j(u\cdot\na u) | \D_j u\bigr)+\bigl(\D_j[a(\Delta u-\na \pi)] | \D_j u\bigr).
\eeno
The next idea is to apply the commutative estimate for
\beno
[\D_j; a] (\Delta u-\na \pi) \eeno
and the smallness condition for $a$ in the critical Besov space, $\dB^{\f32}_{2,\infty},$ to deal with
the term $\bigl(\D_j[a(\Delta u-\na \pi)] | \D_j u\bigr)$ as a perturbation term of the left-hand side.
To go through this process, one needs $a\in L^\infty\cap B^{\f32}_{2,\infty}.$ Nevertheless, here we only assume that $a_0$ belongs to the bounded
function space. Hence the aforementioned program does not work for the proof of Theorems \ref{thm1} to \ref{thm3} here and the results in \cite{CZZ16,HPZ2, PZZ1}.

The main idea to the proof  of Theorems \ref{thm1} to \ref{thm3} is motivated by the following lemma and its proof in \cite{BCD}:

\begin{lem}[Lemma 2.64 of \cite{BCD}]\label{S2lem1}
{\sl Let $s$ be a positive real number and $(p,r)$ be in $[1,\infty]^2.$ A constant $C_s$ exits such that
if $(u_j)_{j\in\Z}$ is a sequence of smooth functions where $\sum_{j\in\Z}u_j$ converges to $u$ in $\cS_h'$ and
\beno
N_s\bigl((u_j)_{j\in\Z}\bigr)\eqdefa \bigl\|\bigl(\sup_{|\al|\in \{0,[s]+1\}}2^{j(s-|\al|)}\|\p^\al u_j\|_{L^p}\bigr)_{j\in\Z}\bigr\|_{\ell^r(\Z)}<\infty,
\eeno
then $u$ is in $\dB^s_{p,r}$ and $\|u\|_{\dB^s_{p,r}}\leq C_sN_s\bigl((u_j)_{j\in\Z}\bigr).$}
\end{lem}

More precisely,  we shall explain how to combine the energy method
with the proof of Lemma \ref{S2lem1} in \cite{BCD} to  prove  the following proposition:

\begin{prop}\label{S3prop1}
{\sl Let $(\r, u)$ be a smooth enough solution of \eqref{S1eq1} on $[0,T^\ast[.$ Then under the assumption of \eqref{S1eq2} and \eqref{S1eq3},
 we have \eqref{S1eq14} and
\beq
\|u\|_{\wt{L}^\infty_t(B^{\f12})}+\|\na u\|_{\wt{L}^2_t(B^{\f12})}\leq C\|u_0\|_{B^{\f12}}\quad\mbox{for}\ \ t<T^\ast.
 \label{S3eq7}
 \eeq
}
\end{prop}

\begin{proof} We first deduce from the classical theory on transport equation and \eqref{S1eq2} that
there holds \eqref{S1eq14} for $t<T^\ast.$

We now consider the coupled system of $(u_j,\na\pi_j)$ as follows:
\begin{equation}\label{S3eq1}
\left\{
\begin{array}{ll}
\rho\bigl(\p_tu_j+u\cdot\na u_j\bigr)-\Delta u_j+\na \pi_j=0, \\
\textrm{div} u_j=0,\\
u_j|_{t=0}=\D_ju_0.
\end{array}
\right.
\end{equation}
Then we deduce from  the uniqueness of local smooth solution to \eqref{S1eq1} that
\beq\label{S3eq1a}
u=\sum_{j\in\Z} u_j,\andf \na\pi=\sum_{j\in\Z}\na\pi_j.
\eeq

By taking $L^2$ inner product of the momentum equation of \eqref{S3eq1} with $u_j$ and using
the transport equation of \eqref{S1eq1}, we write
\beno
\f12\f{d}{dt}\int_{\R^3}\rho |u_j|^2\,dx+\int_{\R^3}|\na u_j|^2\,dx=0.
\eeno
Integrating the above equation over $[0,t]$ leads to
\beno
\f12\bigl\|\sqrt{\rho}u_j(t)\bigr\|_{L^2}^2+\|\na u\|_{L^2_t(L^2)}^2=\f12\bigl\|\sqrt{\rho_0}\D_j u_0\bigr\|_{L^2}^2,
\eeno
and thus there holds
\beq \label{S3eq2}
\|u_j\|_{L^\infty_t(L^2)}+\|\na u\|_{L^2_t(L^2)}\leq C\|\D_j u_0\|_{L^2}\leq Cd_j2^{-\f{j}2}\|u_0\|_{B^{\f12}}.
\eeq

Whereas taking $L^2$ inner product of the momentum equation of \eqref{S3eq1} with $\p_tu_j$ gives
\beq\label{S3eq2a}
\begin{split}
\f12\f{d}{dt}\|\na u_j(t)\|_{L^2}^2+\bigl\|\sqrt{\rho}\p_tu_j\bigr\|_{L^2}^2=&-\int_{\R^3}\r u\cdot\na u_j | \p_t u_j\,dx\\
\leq & C\|u\|_{L^3}\|\na u_j\|_{L^6}\bigl\|\sqrt{\rho}\p_tu_j\bigr\|_{L^2}\\
\leq &C\|u\|_{\dH^{\f12}}\|\na^2u_j\|_{L^2}\bigl\|\sqrt{\rho}\p_tu_j\bigr\|_{L^2}.
\end{split}
\eeq
On the the other hand, the momentum equation of \eqref{S3eq1} can be reformulated as
\beq \label{S3eq2b}
\left\{
\begin{array}{ll}
-\Delta u_j+\na \pi_j=-
\rho\p_tu_j-\r u\cdot\na u_j, \\
\textrm{div} u_j=0,
\end{array}
\right.
\eeq
from which and the law of product \eqref{S3eq0}, we infer
\beq \label{S3eq3}
\begin{split}
\|\na^2u_j\|_{L^2}+\|\na\pi_j\|_{L^2}\leq &C\bigl(\|\sqrt{\r}\p_tu_j\|_{L^2}+\|u\cdot\na u_j\|_{L^2}\bigr)\\
\leq &C\bigl(\|\sqrt{\r}\p_tu_j\|_{L^2}+\|u\|_{\dot H^{\f12}}\|\na^2 u_j\|_{L^2}\bigr).
\end{split}
\eeq
Let us denote
\beq \label{S3eq4}
T^\star_1\eqdefa \sup\bigl\{\ t<T^\ast:\ \|u\|_{L^\infty_t(\dot H^{\f12})} \leq c_1\ \bigr\}.
\eeq
Then for $c_1$ being so small that $Cc_1\leq \f12$ and for $t\leq T^\star_1,$ we have
\beq \label{S3eq5}
\|\na^2u_j\|_{L^2}+\|\na\pi_j\|_{L^2}\leq C\|\sqrt{\r}\p_tu_j\|_{L^2}.
\eeq
Inserting the above inequality into \eqref{S3eq2a} yields
\beno
\f12\f{d}{dt}\|\na u_j(t)\|_{L^2}^2+\bigl\|\sqrt{\rho}\p_tu_j\bigr\|_{L^2}^2
\leq  C\|u\|_{\dot H^{\f12}}\bigl\|\sqrt{\rho}\p_tu_j\bigr\|_{L^2}^2\quad \mbox{for}\ \ t\leq T^\star_1.
\eeno
Then for $c_1$ in \eqref{S3eq4} being so small that $Cc_1\leq \f12,$ we achieve
\beno
\f{d}{dt}\|\na u_j(t)\|_{L^2}^2+\bigl\|\sqrt{\rho}\p_tu_j\bigr\|_{L^2}^2\leq 0 \quad \mbox{for}\ \ t\leq T^\star_1,
\eeno
from which and \eqref{S3eq5}, we deduce that
\beq
\label{S3eq6}
\begin{split}
\|\na u_j\|_{L^\infty_t(L^2)}+&\bigl\|(\sqrt{\rho}\p_tu_j,\na^2u_j,\na\pi_j)\bigr\|_{L^2_t(L^2)}\\
\leq& C\|\na u_j(0)\|_{L^2}
\leq C\|\na\D_j u_0\|_{L^2}\leq Cd_j2^{\f{j}2}\|u_0\|_{B^{\f12}}.
\end{split}
\eeq

In view of \eqref{S3eq1a}, \eqref{S3eq2} and \eqref{S3eq6}, we get, by applying Bernstein inequality,   that
\beno
\begin{split}
\|\D_ju\|_{L^\infty_t(L^2)}+\|\na\D_ju\|_{L^2_t(L^2)}
\lesssim &\sum_{j'\geq j}\bigl(\|\D_{j}u_{j'}\|_{L^\infty_t(L^2)}+\|\na\D_{j}u_{j'}\|_{L^2_t(L^2)}\bigr)\\
&+2^{-j}\sum_{j'\leq j}\bigl(\|\na\D_{j}u_{j'}\|_{L^\infty_t(L^2)}+\|\na^2\D_{j}u_{j'}\|_{L^2_t(L^2)}\bigr)\\
\lesssim &\sum_{j'\geq j}\bigl(\|u_{j'}\|_{L^\infty_t(L^2)}+\|\na u_{j'}\|_{L^2_t(L^2)}\bigr)\\
&+2^{-j}\sum_{j'\leq j}\bigl(\|\na u_{j'}\|_{L^\infty_t(L^2)}+\|\na^2 u_{j'}\|_{L^2_t(L^2)}\bigr)\\
\lesssim &d_j2^{-\f{j}2}\|u_0\|_{B^{\f12}}.
\end{split}
\eeno
This implies that \eqref{S3eq7} holds for $t\leq T^\star_1.$
Then taking $\e_0$ in \eqref{S1eq3}  so small that $C\|u_0\|_{B^{\f12}}\leq C\e_0\leq \f{c_1}2,$ for $c_1$ given by \eqref{S3eq4}, we get,
by using a continuous argument, that $T^\star_1$ determined by \eqref{S3eq4} equals any number smaller than $T^\ast.$ This in turn shows \eqref{S3eq7}. \end{proof}

\begin{rmk}
We remark that similar idea was first used by Hmidi and Keraani \cite{HK} for two dimensional incompressible
Euler system, which also works (without change) for the transport diffusion equation. In the inhomogeneous context,
similar idea was used by Liao and the author \cite{LZ2} in order to propagate fractional Besov regularities
for the velocity field of  the two dimensional incompressible inhomogeneous Navier-Stokes system.
\end{rmk}

\begin{rmk}
By virtue of \eqref{S3eq7}, we deduce from the classical theory of inhomogeneous
Navier-Stokes system that $T^\ast=\infty.$
\end{rmk}

Along the same line to the proof of Proposition \ref{S3prop1}, we can also show that
\begin{prop}\label{S3prop2}
{\sl Under the assumptions of Proposition \ref{S3prop1}, we have
\beq
\label{S3eq15}
\|\sqrt{t}\na u\|_{\wt{L}^\infty_t(B^{\f12})}+\|\sqrt{t}u_t\|_{\wt{L}^2_t(B^{\f12})}\leq C\|u_0\|_{B^{\f12}}\quad\mbox{for any} \ t>0. \eeq
}\end{prop}
and
\begin{prop}\label{S3prop3}
{\sl Under the assumptions of Proposition \ref{S3prop1}, we have
 \beq
\label{S3eq33}
 \|tD_tu\|_{\wt{L}^\infty_t(B^{\f12})}+\|t\na D_tu\|_{\wt{L}^2_t(B^{\f12})}\leq C\|u_0\|_{B^{\f12}}\quad\mbox{for any} \ t>0. \eeq
}\end{prop}

The complete proof of the above two propositions will be presented in Section \ref{Sect3}.

Now we are in a position to complete the proof of Theorem \ref{thm1}.

\begin{proof}[Proof  of Theorem \ref{thm1}]
By mollifying the initial data $(\rho_0,u_0),$ we deduce from
the classical theory of inhomogeneous incompressible Navier-Stokes system that (\ref{S1eq1}) has a unique local  solution
$(\rho^\e, u^\e)$ on $[0,T^\ast_\e[.$ Moreover, we can show that \eqref{S3eq7},  \eqref{S3eq15} and \eqref{S3eq33}  hold for 
 $(\r^\e, u^\e).$ In particular,  Proposition \ref{S3prop1} ensures $T^\ast_\e=\infty$ provided that
$\e_0$ is small enough in \eqref{S1eq3}.
Then exactly along the same line to proof of Theorem 1.2 in \cite{PZZ1}, we can complete the existence part of Theorem
\ref{thm1} by using the uniform estimates \eqref{S3eq7},  \eqref{S3eq15}, \eqref{S3eq33} and \eqref{S3eq38}
for $(\rho^\e, u^\e)$ and a standard compactness argument, which we omit details here.  In order to
prove \eqref{S1eq7}, it remains to show that
\beq\label{S3eq37}
\|tu_t\|_{\wt{L}^\infty_t(B^{\f12})}\leq C\|u_0\|_{B^{\f12}}.
\eeq
Indeed it follows from the law of product in Besov spaces and \eqref{S3eq15}, \eqref{S3eq33} that
\beq\label{S2eq1}
\begin{split}
\|tu_t\|_{\wt{L}^\infty_t(B^{\f12})}\leq &\|t D_tu\|_{\wt{L}^\infty_t(B^{\f12})}+C\|\sqrt{t}u\|_{\wt{L}^\infty_t(B^{\f32})}\|\|\sqrt{t}\na u\|_{\wt{L}^\infty_t(B^{\f12})}\\
\leq &C\|u_0\|_{B^{\f12}}.
\end{split}
\eeq
This completes the proof of Theorem \ref{thm1}.
\end{proof}

In order to prove Theorem \ref{thm2},
we denote $u=v+e^{t\D}u_0.$ Then by virtue of \eqref{S4eq1}, $(a,v, \na\pi)$ verifies
\beq \label{S4eq2}
\left\{
\begin{array}{ll}
\p_ta+\bigl(v+e^{t\D}u_0\bigr)\cdot\na a=0,\\
\p_tv+\bigl(v+e^{t\D}u_0\bigr)\cdot\na v+v\cdot\na e^{t\D} u_0-(1+a)\bigl(\Delta v-\na \pi)\\
\qquad=-e^{t\D}u_0\cdot\na e^{t\D}u_0+a\D e^{t\D}u_0,\\
\dive v=0,\\
v|_{t=0}=0.
\end{array}
\right.
\eeq
Now let $(a,v,\na\pi)$ be a smooth enough solution of \eqref{S4eq2} on $[0,T^\ast[,$ we construct $(v_j, \na\pi_j)$ via
\beq\label{S4eq3}
\left\{
\begin{array}{ll}
\r\p_tv_j+\r\bigl(v+e^{t\D}u_0\bigr)\cdot\na v_j+\r v_j\cdot\na e^{t\D} u_0-\Delta v_j+\na \pi_j\\
\quad\ =\r F_j+\r a\D e^{t\D}\D_j u_0 \with\\
F_j=\begin{pmatrix} -e^{t\D}u_0\cdot\na e^{t\D}\D_ju_0^\h\\
-e^{t\D}\D_j u_0^\h\cdot\na_\h e^{t\D}u_0^3+e^{t\D}u_0^3e^{t\D}\D_j\dive_\h u_0^\h   \end{pmatrix},\\
\dive v_j=0\\
v_j|_{t=0}=0.
\end{array}
\right.
\eeq
Here and in  Section \ref{Sect3}, we always denote $v^\h\eqdefa (v_1,v_2), v=(v^\h,v^3)$ and $\r\eqdefa\f1{1+a}.$

Due to uniqueness result for smooth enough solution of \eqref{S4eq2} and $\dive u_0=0,$ we have
\beq \label{S4eq3a}
v=\sum_{j\in\Z} v_j \andf \na\pi=\sum_{j\in\Z}\na\pi_j.
\eeq

Then the proof of Theorem \ref{thm2} consists of the following three propositions,  the proof of which will
be presented in Section \ref{Sect4}:

\begin{prop}\label{S4prop1}
{\sl Let $(\r,v)$ be a smooth enough solution of \eqref{S4eq2} on $[0,T^\ast[.$ Then under the assumption of \eqref{small1}, we have
\beq\label{S4eq10}
\|v\|_{\wt{L}^\infty(B^{\f12})}+\|\na v\|_{\wt{L}^2_t(B^{\f12})}\leq  C \eta\quad
\mbox{for}\quad t<T^\ast,
\eeq where $\eta$ is given by \eqref{small1}. }
\end{prop}

 \begin{prop}\label{S4prop2}
 {\sl Under the assumption of Proposition \ref{S4prop1}, we have
 \beq\label{S4eq18}
 \|\sqrt{t}\na v\|_{\wt{L}^\infty(B^{\f12})}+\|\sqrt{t} v_t\|_{\wt{L}^2_t(B^{\f12})}
 \leq C\eta\quad
 \mbox{for any }\quad t>0.
 \eeq }
 \end{prop}

 \begin{prop}\label{S4prop3}
 {\sl Under the assumptions of  Proposition \ref{S4prop1},  for any $t>0,$  we have
 \beq
 \label{S4eq21}
 \|tD_tv\|_{\wt{L}^\infty_t(B^{\f12})}+\|t\na D_tv\|_{\wt{L}^2_t(B^{\f12})}
 \leq C\eta. \eeq }
 \end{prop}

\begin{proof}[Proof of Theorem \ref{thm2}] With the {\it a priori} estimates \eqref{S4eq10}, \eqref{S4eq18} and \eqref{S4eq21},
we can follow the same line as the proof of Theorem 1  of \cite{HPZ2} to complete the existence part of Theorem \ref{thm2}. Moreover,
there holds \eqref{S1eq8}.
\end{proof}

\begin{rmk} We emphasize that we crucially used the divergence free condition of $u_0$ in the construction of $(v_j,\na\pi_j)$ in \eqref{S4eq3}.
We comment that the proof of Theorem \ref{thm2} here is more concise than that of Theorem 1  in \cite{HPZ2},
where the authors first write the integral formulation for the velocity field $u$ of \eqref{S4eq1}, then
perform the estimate of $u^\h$ and finally the estimate of $u^3$ through the application of the maximal regularity estimate for heat-semi-group.
 I am not sure if we can go through the proof of  Theorem \ref{thm2}
by performing energy estimate for $u_j^\h$ and then for the energy estimate of $u_j^3$ for solutions $(u_j,\na\pi_j)$ of \eqref{S3eq1}.
\end{rmk}

Finally, let us turn to the proof of Theorem \ref{thm3}, which relies on the following propositions:

\begin{prop}\label{S5prop1}
{\sl Let $(\r, u)$ be a smooth enough solution of \eqref{S1eq1} on $[0, T^\ast[.$ Then for $\ga\in ]0,1/4],$ there a positive constant $c_\ga$
so that \beq \label{S5eq12} T^\ast> T_\ga\eqdefa c_\ga \|u_0\|_{\dot H^{\f12+2\ga}}^{-\f1\ga},\eeq
and for $t\leq T_\ga,$ there holds
 \beq \label{S5eq11}
 \|u\|_{\wt{L}^\infty_t(\dH^{\f12+2\ga})}+\|\na u\|_{{L}^2_t(\dH^{\f12+2\ga})} \leq C  \|u_0\|_{\dH^{\f12+2\ga}}.
\eeq}
\end{prop}

\begin{prop}\label{S5prop2}
{\sl Under the assumptions of Proposition \ref{S5prop1}, we have
 \beq \label{S5eq10}
    \|\sqrt{t}\na u\|_{\wt{L}^\infty_t(\dH^{\f12+2\ga})}+\|\sqrt{t}u_t\|_{L^2_t(\dH^{\f12+2\ga})}\leq
   C \|u_0\|_{\dH^{\f12+2\ga}} \quad\mbox{for}\ \ t\leq T_\ga.
  \eeq}
  \end{prop}

 \begin{prop}\label{S5prop3}
{\sl Under the assumptions of Proposition \ref{S5prop1}, we have
 \beq
\label{S5eq21}
 \|tD_tu\|_{\wt{L}^\infty_t(\dH^{\f12+2\ga})}+\|t\na D_tu\|_{\wt{L}^2_t(\dH^{\f12+2\ga})}\leq C\|u_0\|_{\dH^{\f12+2\ga}}
\quad\mbox{for} \ \ t\leq T_\ga.
 \eeq
}\end{prop}

 By summing up Proposition \ref{S5prop1} to Proposition \ref{S5prop3}, we conclude the proof of Theorem \ref{thm3}.
 The detailed proof of Propositions \ref{S5prop1} to \ref{S5prop3} will be presented in Section \ref{Sect5}.

 For the convenience of the readers, we shall collect some basic facts on Littlewood-Paley theory in the Appendix \ref{apA}.

\medskip

 \setcounter{equation}{0}
 \section{The proof of Theorem \ref{thm1}}\label{Sect3}

 The purpose of this section is to present the proof of Propositions \ref{S3prop2} to \ref{S3prop3}. We first deduce
 the following
 corollary  from the proof of Proposition \ref{S3prop1} in Section \ref{Sect2}.

\begin{col}
{\sl Let $(u_j,\na\pi_j)$ be dertermined by \eqref{S3eq1}, we have
\beq \label{S3eq9} \|\sqrt{t}\na u_j\|_{L^\infty_t(L^2)}+
\bigl\|\sqrt{t}(\p_tu_j,\na^2u_j,\na\pi_j)\bigr\|_{L^2_t(L^2)}\leq Cd_j2^{-\f{j}2}\|u_0\|_{B^{\f12}}\quad\mbox{for any }\ t>0.
\eeq}
\end{col}

\begin{proof} Indeed thanks to \eqref{S3eq2a}, we have
\beno
\begin{split}
\f12\f{d}{dt}\|\na u_j(t)\|_{L^2}^2+\bigl\|\sqrt{\rho}\p_tu_j\bigr\|_{L^2}^2\leq &C\|u\|_{L^\infty}\|\na u_j\|_{L^2}\|\sqrt{\r}\p_t u_j\|_{L^2}\\
\leq & C\|u\|_{B^{\f32}}^2\|\na u_j\|_{L^2}^2+\f12\|\sqrt{\r}\p_t u_j\|_{L^2}^2.
\end{split}
\eeno
Multiplying the above inequality by $t$ gives rise to
\beno
\f{d}{dt}\bigl(t\|\na u_j(t)\|_{L^2}^2\bigr)+t\bigl\|\sqrt{\rho}\p_tu_j\bigr\|_{L^2}^2\leq \|\na u_j\|_{L^2}^2
+C\|u\|_{B^{\f32}}^2t\|\na u_j\|_{L^2}^2.
\eeno
Applying Gronwall's inequality and then using \eqref{S3eq7} and \eqref{S3eq2} that
\beq \label{S3eq8}
\begin{split}
\|\sqrt{t}\na u_j\|_{L^\infty_t(L^2)}^2+\|\sqrt{t}\p_tu_j\|_{L^2_t(L^2)}^2\leq &\|\na u_j\|_{L^2_t(L^2)}^2\exp\bigl(C\|u\|_{L^2_t(B^{\f32})}^2\bigr)\\
\leq & Cd_j^22^{-{j}}\|u_0\|_{B^{\f12}}^2,
\end{split}
\eeq
which together with \eqref{S3eq5} implies \eqref{S3eq9}\end{proof}

\begin{proof}[Proof of Proposition \ref{S3prop2}] Applying $\p_t$ to the momentum equation of \eqref{S3eq1} yields
\beno
\r\p_t^2u_j+\r u\cdot\na \p_t u_j-\D \p_tu_j+\na\p_t\pi_j=-\r_tD_tu_j-\r u_t\cdot\na u_j.
\eeno
Here and in all that follows, we always denote $D_t\eqdefa \p_t+u\cdot\na$ to be the material derivative.

By taking $L^2$ inner product of the above equation with $\p_tu_j$ and making use of the transport equation of
\eqref{S1eq1}, we obtain
\beq \label{S3eq10}
\f12\f{d}{dt}\int_{\R^3}\r|\p_tu_j|^2\,dx+\|\na\p_tu_j\|_{L^2_t(L^2)}^2=-\int_{\R^3}\r_tD_tu_j | \p_tu_j\,dx-
\int_{\R^3}\r u_t\cdot\na u_j  | \p_tu_j\,dx.
\eeq
Let us now handle term by term above.

\noindent\underline{$\bullet$ \it  Estimate for $\int_{\R^3}\r_t\p_tu_j | \p_tu_j\,dx.$ }

By virtue of the transport equation of \eqref{S1eq1}, we get, by using integration by parts, that
\beno
\begin{split}
-\int_{\R^3}\r_t\p_tu_j | \p_tu_j\,dx=&\int_{\R^3}u\cdot\na\r |\p_tu_j|^2\,dx\\
=&-2\int_{\R^3}\r u \cdot\na \p_tu_j | \p_tu_j\,dx,
\end{split}
\eeno
which leads to
\beno
\begin{split}
\bigl|\int_{\R^3}\r_t\p_tu_j | \p_tu_j\,dx\bigr|\leq &C\|u\|_{L^\infty}\|\sqrt{\r}\p_tu_j\|_{L^2}\|\na\p_tu_j\|_{L^2}\\
\leq &C\|u\|_{B^{\f32}}^2\|\sqrt{\r}\p_tu_j\|_{L^2}^2+\f16\|\na\p_tu_j\|_{L^2}^2.
\end{split}
\eeno

\noindent\underline{$\bullet$ \it  Estimate for $\int_{\R^3}\r_tu\cdot\na u_j | \p_tu_j\,dx.$ }

Once again we get, by using integration by parts, that
\beno
\begin{split}
-\int_{\R^3}\r_tu\cdot\na u_j | \p_tu_j\,dx=&-\int_{\R^3}\bigl(\r u\cdot\na u\bigr)\cdot\na u_j | \p_tu_j\,dx\\
&-\int_{\R^3}\r (u\otimes u):\na^2 u_j | \p_tu_j\,dx-\int_{\R^3} u\cdot\na u_j | \r u\cdot\na \p_tu_j\,dx.
\end{split}
\eeno
It follows from \eqref{S3eq5} that
\beno
\begin{split}
\bigl|\int_{\R^3}\bigl(\r u\cdot\na u\bigr)\cdot\na u_j | \p_tu_j\,dx\bigr|\leq &C\|u\|_{L^\infty}\|\na u\|_{L^3}\|\na u_j\|_{L^6}\|\sqrt{\r}\p_tu_j\|_{L^2}\\
\leq &C\|u\|_{B^{\f32}}^2\|\na^2u_j\|_{L^2}\|\sqrt{\r}\p_tu_j\|_{L^2}\\
\leq &C\|u\|_{B^{\f32}}^2\|\sqrt{\r}\p_tu_j\|_{L^2}^2.
\end{split}
\eeno
The same estimate holds for $\int_{\R^3}\r (u\otimes u):\na^2 u_j | \p_tu_j\,dx.$  Moreover, along the same line, we have
\beno
\begin{split}
\bigl|\int_{\R^3} u\cdot\na u_j | \r u\cdot\na \p_tu_j\,dx\bigr|\leq &C\|u\|_{L^\infty}\|\na u_j\|_{L^6}\|u\|_{L^3}\|\na\p_tu_j\|_{L^2}\\
\leq &C\|u\|_{B^{\f12}}\|u\|_{B^{\f32}}\|\na^2u_j\|_{L^2}\|\na\p_tu_j\|_{L^2}\\
\leq &C\|u\|_{B^{\f12}}^2\|u\|_{B^{\f32}}^2\|\sqrt{\r}\p_tu_j\|_{L^2}^2+\f16\|\na\p_tu_j\|_{L^2}^2.
\end{split}
\eeno
As a result, it comes out
\beno
\bigl|\int_{\R^3}\r_tu\cdot\na u_j | \p_tu_j\,dx\bigr|\leq C\bigl(1+\|u\|_{B^{\f12}}^2\bigr)\|u\|_{B^{\f32}}^2\|\sqrt{\r}\p_tu_j\|_{L^2}^2+\f16\|\na\p_tu_j\|_{L^2}^2.
\eeno

\noindent\underline{$\bullet$ \it  Estimate for $\int_{\R^3}\r u_t\cdot\na u_j  | \p_tu_j\,dx.$ }

\beno
\begin{split}
\bigl|\int_{\R^3}\r u_t\cdot\na u_j  | \p_tu_j\,dx\bigr|\leq &C\|u_t\|_{L^2}\|\na u_j\|_{L^3}\|\p_tu_j\|_{L^6}\\
\leq &C\|u_t\|_{L^2}\|\na u_j\|_{L^2}^{\f12}\|\na^2u_j\|_{L^2}^{\f12}\|\na\p_tu_j\|_{L^2}\\
\leq &C\|u_t\|_{L^2}^2\|\na u_j\|_{L^2}\|\na^2u_j\|_{L^2}+\f16\|\na\p_tu_j\|_{L^2}^2\\
\leq &C\|u_t\|_{L^2}^2\|\na u_j\|_{L^2}\|\sqrt{\r}\p_tu_j\|_{L^2}+\f16\|\na\p_tu_j\|_{L^2}^2,
\end{split}
\eeno
where in the last step, we used \eqref{S3eq5}.

Inserting the above estimates into \eqref{S3eq10} gives rise to
\beq \label{S3eq17}
\begin{split}
\f{d}{dt}\int_{\R^3}\r|\p_tu_j|^2\,dx+\|\na\p_tu_j\|_{L^2_t(L^2)}^2\leq C\Bigl(&\bigl(1+\|u\|_{B^{\f12}}^2\bigr)\|u\|_{B^{\f32}}^2\|\sqrt{\r}\p_tu_j\|_{L^2}^2\\
&\quad+\|u_t\|_{L^2}^2\|\na u_j\|_{L^2}\|\sqrt{\r}\p_tu_j\|_{L^2}\Bigr).
\end{split}
\eeq
Multiplying the above inequality by $t$ and using \eqref{S3eq7} yields
\beno
\begin{split}
\f{d}{dt}\|\sqrt{t\r}\p_tu_j\|_{L^2}^2+\|\sqrt{t}\na\p_tu_j\|_{L^2_t(L^2)}^2
\leq &\|\sqrt{\r}\p_tu_j\|_{L^2}^2+C\|t^{\f14}u_t\|_{L^2}^2\|\na u_j\|_{L^2}^2\\
&+C\bigl(\|u\|_{B^{\f32}}^2+\|t^{\f14}u_t\|_{L^2}^2\bigr)\|\sqrt{t\r}\p_tu_j\|_{L^2}^2.
\end{split}
\eeno
By applying Gronwall's inequality and using \eqref{S1eq14}, we achieve
\beq \label{S3eq17a}
\begin{split}
\|\sqrt{t}\p_tu_j\|_{L^\infty_t(L^2)}^2+\|\sqrt{t}\na\p_tu_j\|_{L^2_t(L^2)}^2&\leq C\exp\Bigl(\int_0^t\bigl(\|u\|_{B^{\f32}}^2+\|t^{\f14}u_t\|_{L^2}^2\bigr)\,dt'\Bigr)
\\
&\times \bigl(\|\sqrt{\r}\p_tu_j\|_{L^2_t(L^2)}^2
+\|t^{\f14}u_t\|_{L^2_t(L^2)}^2\|\na u_j\|_{L^\infty_t(L^2)}^2\bigr).
\end{split}
\eeq On the other hand, we deduce from \eqref{S3eq6} and \eqref{S3eq9} that
\beno
\begin{split}
\|t^{\f14}\p_tu_j\|_{L^2_t(L^2)}\leq &\|\p_tu_j\|_{L^2_t(L^2)}^{\f12}\|\sqrt{t}\p_tu_j\|_{L^2_t(L^2)}^{\f12}\\
\leq & Cd_j\|u_0\|_{B^{\f12}},
\end{split}
\eeno
so that it follows from \eqref{S3eq1a}
\beq \label{S3eq12}
\|t^{\f14}\p_tu\|_{L^2_t(L^2)}\leq \sum_{j\in\Z}\|t^{\f14}\p_tu_j\|_{L^2_t(L^2)}\leq C\|u_0\|_{B^{\f12}}.
\eeq
Therefore, by virtue of \eqref{S3eq7}, \eqref{S3eq6} and \eqref{S3eq12}, we obtain
\beq \label{S3eq13}
\begin{split}
\|\sqrt{t}\p_tu_j\|_{L^\infty_t(L^2)}+\|\sqrt{t}\na\p_tu_j\|_{L^2_t(L^2)}\leq & Cd_j2^{\f{j}2}\exp\bigl(C\|u_0\|_{B^{\f12}}^2\bigr)\|u_0\|_{B^{\f12}}\\
\leq & Cd_j2^{\f{j}2}\|u_0\|_{B^{\f12}},
\end{split}
\eeq
from which and \eqref{S3eq5}, we achieve
\beq \label{S3eq14}
\|\sqrt{t}\na^2u_j\|_{L^\infty_t(L^2)}\leq Cd_j2^{\f{j}2}\|u_0\|_{B^{\f12}}.
\eeq
With \eqref{S3eq9}, \eqref{S3eq13} and \eqref{S3eq14}, we get, by a similar derivation of \eqref{S3eq7}, that \eqref{S3eq15} holds for  any
$t>0.$
\end{proof}

\begin{col}\label{S3col2}
{\sl Under the assumptions of Proposition \ref{S3prop1}, for any $t>0,$ we have
\beq \label{S3eq32}
\begin{split}
\|t D_tu_j\|_{L^\infty_t(L^2)}+\|t \na D_tu_j\|_{L^2_t(L^2)}
\leq  Cd_j2^{-\f{j}2}\|u_0\|_{B^{\f12}},
\end{split}
\eeq
and
\beq \label{S3eq16}
\|\sqrt{t}\na^2u\|_{L^2_t(L^3)}\leq C\|u_0\|_{B^{\f12}}.
\eeq}
\end{col}

\begin{proof} By multiplying \eqref{S3eq17} by $t^2,$ we obtain
\beno
\begin{split}
\f{d}{dt}\|t\sqrt{\r}\p_tu_j\|_{L^2}^2+\|t\na\p_tu_j\|_{L^2_t(L^2)}^2
\leq &2\|\sqrt{t\r}\p_tu_j\|_{L^2}^2+C\|t^{\f14}u_t\|_{L^2}^2\|\sqrt{t}\na u_j\|_{L^2}^2\\
&+C\bigl(\|u\|_{B^{\f32}}^2+\|t^{\f14}u_t\|_{L^2}^2\bigr)\|t\sqrt{\r}\p_tu_j\|_{L^2}^2.
\end{split}
\eeno
Applying Gronwall's inequality gives
\beq \label{S3eq18}
\begin{split}
\|t\p_tu_j\|_{L^\infty_t(L^2)}^2+&\|t\na\p_tu_j\|_{L^2_t(L^2)}^2\leq
\exp\Bigl(C\int_0^t\bigl(\|u\|_{B^{\f32}}^2+\|t^{\f14}u_t\|_{L^2}^2\bigr)\,dt'\Bigr)\\
&\qquad\times \bigl(\|\sqrt{t}\p_tu_j\|_{L^2_t(L^2)}^2+C\|t^{\f14}u_t\|_{L^2_t(L^2)}^2\|\sqrt{t}\na u_j\|_{L^\infty_t(L^2)}^2\bigr),
\end{split}
\eeq
which together with  \eqref{S3eq9} and \eqref{S3eq12} implies
\beq \label{S3eq18}
\|t\p_tu_j\|_{L^\infty_t(L^2)}^2+\|t\na\p_tu_j\|_{L^2_t(L^2)}^2\leq
  Cd_j^22^{-j}\|u_0\|_{B^{\f12}}^2.
\eeq
Then by using \eqref{S3eq15} and \eqref{S3eq9} once again, we get
\beno
\begin{split}
\|t D_tu_j\|_{L^\infty_t(L^2)}\leq &\|t\p_tu_j\|_{L^\infty_t(L^2)}+\|\sqrt{t}u\|_{L^\infty_t(B^{\f32})}\|\sqrt{t}\na u_j\|_{L^\infty_t(L^2)}\\
\leq & Cd_j2^{-\f{j}2}\|u_0\|_{B^{\f12}}.
\end{split}
\eeno
Whereas by applying the law of product in Sobolev space, \eqref{S3eq0}, and using \eqref{S3eq5} gives rise to
\beno
\begin{split}
\|t \na D_tu_j\|_{L^2_t(L^2)}\leq &\|t\na\p_t u_j\|_{L^2_t(L^2)}+\|u\|_{L^2_t(B^{\f32})}\|t\na^2u_j\|_{L^\infty_t(L^2)}\\
\leq &\|t\na\p_t u_j\|_{L^2_t(L^2)}+C\|u\|_{L^2_t(B^{\f32})}\|t\p_tu_j\|_{L^\infty_t(L^2)}\\
\leq & Cd_j2^{-\f{j}2}\|u_0\|_{B^{\f12}},
\end{split}
\eeno
which implies \eqref{S3eq32}.

On the other hand, note from \eqref{S1eq1} that
\beno
\left\{
\begin{array}{ll}
-\Delta u+\na \pi=-\rho(\p_tu+u\cdot\na u), \\
\textrm{div} u=0,
\end{array}
\right.
\eeno
from which and regularity theory of Stokes operator, we deduce from \eqref{S3eq7} and \eqref{S3eq15} that
\beno
\begin{split}
\|\sqrt{t}\na^2u\|_{L^2_t(L^3)}\leq &C\bigl(\|\sqrt{t}\p_tu\|_{L^2_t(L^3)}+\|u\|_{L^2_t(L^\infty)}\|\sqrt{t}\na u\|_{L^\infty_t(L^3)}\bigr)
\\
\leq &C\bigl(\|\sqrt{t}\p_tu\|_{L^2_t(B^{\f12})}+\|u\|_{L^2_t(B^{\f32})}\|\sqrt{t}\na u\|_{L^\infty_t(B^{\f12})}\bigr)\\
\leq &C\|u_0\|_{B^{\f12}},
\end{split}
\eeno
which leads to \eqref{S3eq16}.
\end{proof}

\begin{proof}[Proof of Proposition \ref{S3prop3}]
It is easy to observe that
\beno
\begin{split}
[D_t; \na]f=&-\na u\cdot\na f,\andf\\
[D_t; \D]f=&-\D u\cdot\na f-2\sum_{i=1}^3\p_iu\cdot\na\p_i f.
\end{split}
\eeno
Then we get, by applying the operator $D_t$ to the momentum equation of \eqref{S3eq1}, that
\beq \label{S3eq19}
\r D_t^2u_j-\D D_tu_j+\na D_t\pi_j=-\D u\cdot\na u_j-2\sum_{i=1}^3\p_iu\cdot\na\p_i u_j+\na u\cdot\na\pi_j\eqdefa f_j.
\eeq
Moreover, due to $\dive u_j=0,$ we have
\beno
\dive D_tu_j=\sum_{i}^3\p_iu\cdot\na u_j^i.
\eeno
Then we get, by taking space divergence operator to \eqref{S3eq19}, that
\beno
\D D_t\pi_j=\sum_{i}^3\D \left(\p_iu\cdot\na u_j^i\right)-\dive\left(\r D_t^2u_j\right)+\dive f_j,
\eeno
from which, we infer
\beq \label{S3eq20}
\|\na D_t\pi_j\|_{L^2}\leq C\bigl(\bigl\|\na{\rm Tr}(\na u \na u_j)\bigr\|_{L^2}+\|\r D_t^2u_j\|_{L^2}+\|f_j\|_{L^2}\bigr),
\eeq
which together with \eqref{S3eq19} ensures that $\|\na^2 D_tu_j\|_{L^2}$ shares the same estimate.

Now by taking $L^2$ inner product of \eqref{S3eq19} with $D_t^2u_j,$ we achieve
\beno
\|\sqrt{\r}D_t^2u_j\|_{L^2}^2-\int_{\R^3}\D D_tu_j | D_t^2u_j\,dx+\int_{\R^3}\na D_t\pi_j |  D_t^2u_j\,dx=
\int_{\R^3} f_j |  D_t^2u_j\,dx.
\eeno
By using integration by parts, one has
\beno
-\int_{\R^3}\D D_tu_j | D_t^2u_j\,dx=\f12\f{d}{dt}\int_{\R^3}|\na D_tu_j|^2\,dx+\int_{\R^3} \na D_tu_j | [\na; D_t]D_tu_j\,dx,
\eeno
so that there holds
\beq \label{S3eq21}
\begin{split}
\f12\f{d}{dt}\|t\na D_tu_j(t)\|_{L^2}^2+\|t\sqrt{\r}&D_t^2u_j\|_{L^2}^2
 \leq \|\sqrt{t}\na D_tu_j\|_{L^2}^2+t^2\int_{\R^3} f_j |  D_t^2u_j\,dx\\
 &-t^2\int_{\R^3} \na D_tu_j | [\na; D_t]D_tu_j\,dx-t^2\int_{\R^3}\na D_t\pi_j |  D_t^2u_j\,dx.
\end{split}
\eeq

\noindent\underline{$\bullet$ \it  Estimate for $\int_0^t(t')^2\int_{\R^3}f_j |  D_t^2u_j\,dx\,dt'.$ }

Before proceeding,  we notice from \eqref{S3eq2b} that
\beno
\begin{split}
\bigl\|\sqrt{t}(\na^2u_j,\na\pi_j)\bigr\|_{L^2_t(L^6)}\leq & C\bigl(\|
\sqrt{t}\na \p_tu_j\|_{L^2_t(L^2)}+\bigl\|\sqrt{t}u\cdot\na u_j\|_{L^2_t(L^6)}\bigr)\\
\leq &C\bigl(\|\sqrt{t}\na\p_tu_j\|_{L^2_t(L^2)}+\|u\|_{L^2_t(L^\infty)}\|\sqrt{t}\na^2 u_j\|_{L^\infty_t(L^2)}\bigr),
\end{split}
\eeno
which together with \eqref{S3eq7}, \eqref{S3eq13} and \eqref{S3eq14} ensures that
\beq
\label{S3eq22}
\bigl\|\sqrt{t}(\na^2u_j,\na\pi_j)\bigr\|_{L^2_t(L^6)}\leq Cd_j2^{\f{j}2}\|u_0\|_{B^{\f12}}.
\eeq
Then we have
\beno
\begin{split}
\bigl\|t\na{\rm Tr}(\na u \na u_j)\bigr\|_{L^2_t(L^2)}\leq &\|\sqrt{t}\na^2u\|_{L^2_t(L^3)}\|\sqrt{t}\na u_j\|_{L^\infty_t(L^6)}+
\|\sqrt{t}\na u\|_{L^\infty_t(L^3)}\|\sqrt{t}\na^2u_j\|_{L^2_t(L^6)}\\
\leq &C\bigl(\|\sqrt{t}\na^2u\|_{L^2_t(L^3)}\|\sqrt{t}\na^2 u_j\|_{L^\infty_t(L^2)}\\
&\quad\qquad\qquad+
\|\sqrt{t}\na u\|_{L^\infty_t(B^{\f12})}\|\sqrt{t}\na^2u_j\|_{L^2_t(L^6)}\bigr),
\end{split}
\eeno
from which, \eqref{S3eq14}, \eqref{S3eq16} and \eqref{S3eq22}, we infer
\beq\label{S3eq23}
\bigl\|t\na{\rm Tr}(\na u \na u_j)\bigr\|_{L^2_t(L^2)}\leq Cd_j2^{\f{j}2}\|u_0\|_{B^{\f12}}.
\eeq
The same estimate holds for $\|t\D u\cdot\na u_j\|_{L^2_t(L^2)}+\sum_{i=1}^3\|t\p_iu\cdot\na\p_i u_j\|_{L^2_t(L^2)}.$

Furthermore, we deduce from \eqref{S3eq22} that
\beno
\begin{split}
\|t\na u\cdot\na\pi_j\|_{L^2_t(L^2)}\leq &\|\sqrt{t}\na u\|_{L^\infty_t(L^3)}\|\sqrt{t}\na\pi_j\|_{L^2_t(L^6)}\\
\leq &Cd_j2^{\f{j}2}\|\sqrt{t}\na u\|_{L^\infty_t(B^{\f12})}\|u_0\|_{B^{\f12}}.
\end{split}
\eeno
This along with \eqref{S3eq15} ensures that
\beq \label{S3eq24}
\|t f_j\|_{L^2_t(L^2)}\leq Cd_j2^{\f{j}2}\|u_0\|_{B^{\f12}},
\eeq
and thus
\beq\label{S3eq25}
\bigl|\int_0^t(t')^2\int_{\R^3}f_j |  D_t^2u_j\,dx\,dt'\bigr|\leq Cd_j^22^{j}\|u_0\|_{B^{\f12}}^2+\f16\|t\sqrt{\r}D_t^2u_j\|_{L^2_t(L^2)}^2.
\eeq

\noindent\underline{$\bullet$ \it  Estimate for $\int_0^t(t')^2\int_{\R^3}\na D_tu_j | [\na; D_t]D_tu_j\,dx\,dt'.$ }

It is easy to observe that
\beno
\begin{split}
\bigl|\int_{\R^3}\na D_tu_j | [\na; D_t]D_tu_j\,dx\bigr|\leq &\|\na u\|_{L^3}\|\na D_t u_j\|_{L^3}^2\\
\leq &C\|\na u\|_{L^3}\|\na D_t u_j\|_{L^2}\|\na^2 D_t u_j\|_{L^2}.
\end{split}
\eeno
Yet we deduce from \eqref{S3eq20}, \eqref{S3eq23} and \eqref{S3eq24} that
\beq \label{S3eq26}
\begin{split}
\|t\na^2&D_tu_j\|_{L^2_t(L^2)}+\|t\na D_t\pi_j\|_{L^2_t(L^2)}\\
\leq & C\bigl(\bigl\|t\na {\rm Tr}(\na u \na u_j)\bigr\|_{L^2_t(L^2)}+\|t\r D_t^2u_j\|_{L^2_t(L^2)}+\|tf_j\|_{L^2_t(L^2)}\bigr)\\
\leq &C\bigl(d_j2^{\f{j}2}\|u_0\|_{B^{\f12}}+\|t\sqrt{\r} D_t^2u_j\|_{L^2_t(L^2)}\bigr).
\end{split}
\eeq
Then we obtain
\beq \label{S3eq28}
\begin{split}
\bigl|\int_0^t(t')^2\int_{\R^3}\na D_tu_j |& [\na; D_t]D_tu_j\,dx\,dt'\bigr|\leq \f16\|t\sqrt{\r} D_t^2u_j\|_{L^2_t(L^2)}^2\\
&+C\Bigl(d_j^22^{j}\|u_0\|_{B^{\f12}}^2+\int_0^t
\|\na u\|_{B^{\f12}}^2\|t\na D_t u_j\|_{L^2}^2\,dt'\Bigr).
\end{split}
\eeq

\noindent\underline{$\bullet$ \it  Estimate for $\int_0^t(t')^2\int_{\R^3}\na D_t\pi_j |  D_t^2u_j\,dx\,dt'.$ }

We first get, by using integration by parts, that
\beno
\begin{split}
\int_{\R^3}\na D_t\pi_j |  D_t^2u_j\,dx=&-\int_{R^3}D_t\pi_j | \p_t\dive D_tu_j\,dx+\int_{\R^3}\na D_t\pi_j | u\cdot\na D_tu_j\,dx\\
=&-\int_{R^3}D_t\pi_j | \p_t{\rm Tr}(\na u\na u_j )\,dx +\int_{\R^3}\na D_t\pi_j | u\cdot\na D_tu_j\,dx\\
=&\int_{R^3}\na D_t\pi_j | \p_tu\cdot\na u_j\,dx+\int_{R^3}\na D_t\pi_j | \p_tu_j\cdot\na u\,dx\\
&+\int_{\R^3}\na D_t\pi_j | u\cdot\na D_tu_j\,dx.
\end{split}
\eeno
Next we handle term by term above. We first get, by applying product laws in Besov spaces, \eqref{S3eq0}, that
\beno
\begin{split}
\bigl|\int_0^t(t')^2\int_{R^3}\na D_t\pi_j | \p_tu\cdot\na u_j\,dx\,dt'\bigr|\leq &\int_0^t(t')^2\|\na D_t \pi_j\|_{L^2}\|\p_tu\cdot\na u_j\|_{L^2}\,dt'\\
\leq &\|t\na D_t\pi_j\|_{L^2_t(L^2)}\|\sqrt{t}u_t\|_{L^2_t(B^{\f12})}\|\sqrt{t}\na^2 u_j\|_{L^\infty_t(L^2)},
\end{split}
\eeno
which together \eqref{S3eq15}, \eqref{S3eq14} and \eqref{S3eq26} ensures that
\beno
\bigl|\int_0^t(t')^2\int_{R^3}\na D_t\pi_j | \p_tu\cdot\na u_j\,dx\,dt'\bigr|\leq
\f1{18}\|t\sqrt{\r} D_t^2u_j\|_{L^2_t(L^2)}^2+Cd_j^22^{j}\|u_0\|_{B^{\f12}}^2.
\eeno
Along the same line,  we get, by applying \eqref{S3eq15}, \eqref{S3eq13} and \eqref{S3eq26}, that
\beno
\begin{split}
\bigl|\int_0^t(t')^2\int_{R^3}\na D_t\pi_j | \p_tu_j\cdot\na u\,dx\,dt'\bigr|\leq &\|t\na D_t\pi_j\|_{L^2_t(L^2)}\|\sqrt{t}\na\p_tu_j\|_{L^2_t(L^2)}
\|\sqrt{t}\na u\|_{L^\infty_t(B^{\f12})}\\
\leq &
\f1{18}\|t\sqrt{\r} D_t^2u_j\|_{L^2_t(L^2)}^2+Cd_j^22^{j}\|u_0\|_{B^{\f12}}^2,
\end{split}
\eeno
and
\beno
\begin{split}
\bigl|\int_0^t&(t')^2\int_{\R^3}\na D_t\pi_j | u\cdot\na D_tu_j\,dx\,dt'\\
\leq & \int_0^t\|u\|_{L^\infty}\|t\na D_tu_j\|_{L^2}\|t\na D_t\pi_j\|_{L^2}\,dt'\\
\leq &
\f1{18}\|t\sqrt{\r} D_t^2u_j\|_{L^2_t(L^2)}^2+Cd_j^22^{j}\|u_0\|_{B^{\f12}}^2+\int_0^t\|u\|_{L^\infty}^2\|t\na D_tu_j\|_{L^2}^2\,dt'.
\end{split}
\eeno
As a result, it comes out
\beq
\label{S3eq29}
\begin{split}
\bigl|\int_0^t(t')^2\int_{\R^3}\na D_t\pi_j |  D_t^2u_j\,dx\,dt'\bigr|\leq&\f1{6}\|t\sqrt{\r} D_t^2u_j\|_{L^2_t(L^2)}^2+Cd_j^22^{j}\|u_0\|_{B^{\f12}}^2\\
&+C\int_0^t\|u\|_{L^\infty}^2\|t\na D_tu_j\|_{L^2}^2\,dt'.
\end{split}
\eeq

Integrating \eqref{S3eq21} over $[0,t]$ and then inserting \eqref{S3eq25}, \eqref{S3eq28} and \eqref{S3eq29} into the resulting inequality, we achieve
\beno
\begin{split}
\|t\na D_tu_j\|_{L^\infty_t(L^2)}^2+\|t\sqrt{\r}D_t^2u_j\|_{L^2_t(L^2)}^2
 \leq&Cd_j^22^{j}\|u_0\|_{B^{\f12}}^2+\|\sqrt{t}\na D_tu_j\|_{L^2_t(L^2)}^2\\
&+C\int_0^t\bigl(\|\na u\|_{B^{\f12}}^2+\|u\|_{L^\infty}^2\bigr)\|t\na D_tu_j\|_{L^2}^2\,dt'.
\end{split}
\eeno
Note from \eqref{S3eq0}, \eqref{S3eq13} and \eqref{S3eq14} that
\beno
\begin{split}
\|\sqrt{t}\na D_tu_j\|_{L^2_t(L^2)}\leq& \|\sqrt{t}\na \p_tu_j\|_{L^2_t(L^2)}+\|\sqrt{t}u\cdot\na u_j\|_{L^2_t(\dH^1)}\\
\leq &C\bigl(\|\sqrt{t}\na \p_tu_j\|_{L^2_t(L^2)}+\|u\|_{L^2_t(B^{\f32})}\|\sqrt{t}\na u_j\|_{L^\infty_t(\dH^1)}\\
\leq &Cd_j2^{\f{j}2}\|u_0\|_{B^{\f12}}.
\end{split}
\eeno
Then applying Gronwall's inequality and using \eqref{S3eq7} gives rise to
\beq
\label{S3eq30}
\begin{split}
\|t\na D_tu_j\|_{L^\infty_t(L^2)}^2+\|t D_t^2u_j\|_{L^2_t(L^2)}^2 \leq
&Cd_j^22^{j}\|u_0\|_{B^{\f12}}^2\exp\bigl(C\|u\|_{L^2_t(B^{\f32})}^2\bigr)\\
 \leq & Cd_j^22^{j}\|u_0\|_{B^{\f12}}^2, \end{split} \eeq
 which together with \eqref{S3eq26} ensures that
 \beq\label{S3eq40}
 \|t\na^2 D_t u_j\|_{L^2_t(L^2)}\leq  Cd_j2^{\f{j}2}\|u_0\|_{B^{\f12}}.
 \eeq

 Thanks to  \eqref{S3eq32} and \eqref{S3eq30}, \eqref{S3eq40},  we get, by a similar derivation of \eqref{S3eq7} that \eqref{S3eq33} holds for $t>0.$
\end{proof}

\begin{rmk}
We deduce from \eqref{S3eq9} that
\beno
\bigl\|\sqrt{t}(\na u_j, \pi_j)\bigr\|_{L^2_t(L^6)}\leq Cd_j2^{-\f{j}2}\|u_0\|_{B^{\f12}}\quad\mbox{for any }\ t>0,
\eeno
which together with \eqref{S3eq22} ensures that
\beq \label{S3eq38}
\bigl\|\sqrt{t}(\na u, \pi)\bigr\|_{L^2_t(\dB^{\f12}_{6,1})}\leq C\|u_0\|_{B^{\f12}}\quad\mbox{for any }\ t>0.
\eeq
\end{rmk}

\begin{rmk} We remark that the advantage of applying the material derivative, $D_t,$ instead of $\p_t$ to the momentum equation of \eqref{S1eq1}
is that $D_t\r=0.$  Indeed the energy estimate for $D_tu$ was first performed by Hoff in \cite{Hoff} for the isentropic compressible Navier-Stokes system.
For the inhomogeneous incompressible case, similar estimate was first obtained by Liao and the author in \cite{LZ2}.
\end{rmk}

\medskip

\setcounter{equation}{0}
\section{The proof of Theorem \ref{thm2} }\label{Sect4}

The goal of this section is to present the proof of Propositions \ref{S4prop1} to \ref{S4prop3}.

\begin{lem}\label{S4lem1}
{\sl  Let $(a,v,\na\pi)$ be a smooth enough solution of \eqref{S4eq2} on $[0,T^\ast[.$ Let $(v_j,\na\pi_j)$
be determined by \eqref{S4eq3}. Then for $t<T^\ast,$ one has
\beq \label{S4eq5}
 \|v_j\|_{L^\infty_t(L^2)}^2+\|\na v_j\|_{L^2_t(L^2)}^2\leq
  Cd_j^22^{-j}\bigl(\|u_0^\h\|_{B^{\f12}}^2+\|a_0\|_{L^\infty}^2\bigr)\exp\bigl(C\|u_0\|_{B^{\f12}}^2\bigr).
 \eeq}
 \end{lem}

\begin{proof}
We first get, by taking $L^2$ inner product of the momentum equation of \eqref{S4eq3} with $ v_j$ and using the transport equation of \eqref{S4eq2}, that
\beq \label{S4eq4}
\begin{split}
\f12\f{d}{dt}&\|\sqrt{\r}v_j(t)\|_{L^2}^2+\|\na v_j\|_{L^2}^2=-\int_{\R^3}\r v_j\cdot\na e^{t\D}u_0 | v_j\,dx\\
&-\int_{\R^3}\r e^{t\D}u_0\cdot\na e^{t\D}\D_ju_0^\h | v_j^\h\,dx+\int_{\R^3}\r a \D e^{t\D}\D_j u_0 | v_j\,dx\\
&-\int_{\R^3}\r\bigl(e^{t\D}\D_ju_0^\h\cdot\na_\h e^{t\D}u_0^3-e^{t\D}u_0^3e^{t\D}\D_j\dive_\h u_0^\h\bigr)
 | v_j^3\,dx.
 \end{split}
 \eeq
Next let us estimate term by term above.

\noindent\underline{$\bullet$ \it  Estimate for $\int_{\R^3}\r v_j\cdot\na e^{t\D}u_0 | v_j\,dx.$ }

\beno
\begin{split}
\bigl|\int_{\R^3}\r v_j\cdot\na e^{t\D}u_0 | v_j\,dx\bigr|\leq &\|\na e^{t\D}u_0\|_{L^3}\|v_j\|_{L^3}^2\\
\leq &C\|\na e^{t\D}u_0\|_{B^{\f12}}\|v_j\|_{L^2}\|\na v_j\|_{L^2}\\
\leq &C\|\na e^{t\D}u_0\|_{B^{\f12}}^2\|v_j\|_{L^2}^2+\f12\|\na v_j\|_{L^2}^2.
\end{split}
\eeno

\noindent\underline{$\bullet$ \it  Estimate for $\int_{\R^3}\r e^{t\D}u_0\cdot\na e^{t\D}\D_ju_0^\h | v_j^\h\,dx.$ }

\beno
\begin{split}
\bigl|\int_{\R^3}\r e^{t\D}u_0\cdot\na e^{t\D}\D_ju_0^\h | v_j^\h\,dx\big|\leq & C\|e^{t\D}u_0\|_{L^\infty}\|\na e^{t\D}\D_ju_0^\h\|_{L^2}\|v_j\|_{L^2}\\
\leq &C\|e^{t\D}u_0\|_{L^\infty}^2\|v_j\|_{L^2}^2+\f12\|\na e^{t\D}\D_ju_0^\h\|_{L^2}^2.
\end{split}
\eeno

\noindent\underline{$\bullet$ \it  Estimate for $\int_{\R^3}\r a \D e^{t\D}\D_j u_0 | v_j\,dx.$ }

\beno\begin{split}
\bigl|\int_{\R^3}\r a \D e^{t\D}\D_j u_0 | v_j\,dx|\leq& C\|a\|_{L^\infty}\|\D e^{t\D}\D_j u_0\|_{L^2}\|v_j\|_{L^2}\\
\leq& C\|a_0\|_{L^\infty}\|\D e^{t\D}\D_j u_0\|_{L^2}\|v_j\|_{L^2}.
\end{split}
\eeno

\noindent\underline{$\bullet$ \it  Estimate for $\int_{\R^3}\r\bigl(e^{t\D}\D_j u_0^\h\cdot\na_\h e^{t\D}u_0^3-e^{t\D}u_0^3e^{t\D}\D_j\dive_\h u_0^\h\bigr)
 | v_j^3\,dx.$ }

 It follows from the law of product \eqref{S3eq0} that
 \beno
\begin{split}
\bigl|&\int_{\R^3}\r\bigl(-e^{t\D}\D_ju_0^\h\cdot\na_\h e^{t\D}u_0^3-e^{t\D}u_0^3e^{t\D}\D_j\dive_\h u_0^\h\bigr)
 | v_j^3\,dx\bigr|\\
 &\leq \bigl(\bigl\|\na_\h e^{t\D} u_0^3\bigr\|_{B^{\f12}}+\bigl\|e^{t\D}u_0^3\|_{L^\infty}\bigr)\bigl\|\na e^{t\D}\D_j u_0^\h\bigr\|_{L^2}\|v_j\|_{L^2}\\
 &\leq C\bigl\|e^{t\D}u_0^3\|_{B^{\f32}}^2\|v_j\|_{L^2}^2+\f12\bigl\|\na e^{t\D}\D_j u_0^\h\bigr\|_{L^2}.
 \end{split}
 \eeno

 Inserting the above estimates into \eqref{S4eq4} and then applying Gronwall's and  Young's inequalities gives rise to
 \beno
 \begin{split}
 \|v_j&\|_{L^\infty_t(L^2)}^2+\|\na v_j\|_{L^2_t(L^2)}^2\leq  C\exp\bigl(C\bigl\|e^{t\D}u_0\|_{L^2_t(B^{\f32})}^2\bigr)\\
 &\times \Bigl(\bigl\|\na e^{t\D}\D_j u_0^\h\bigr\|_{L^2_t(L^2)}^2+\|a_0\|_{L^\infty}\|\D e^{t\D}\D_j u_0\|_{L^1_t(L^2)}\|v_j\|_{L^\infty_t(L^2)}\Bigr) \\
 \leq & \f12\|v_j\|_{L^\infty_t(L^2)}^2+ C\exp\bigl(C\bigl\|e^{t\D}u_0\|_{L^2_t(B^{\f32})}^2\bigr)\\
 &\times \Bigl(\bigl\|\na e^{t\D}\D_j u_0^\h\bigr\|_{L^2_t(L^2)}^2+\|a_0\|_{L^\infty}^2\|\D e^{t\D}\D_j u_0\|_{L^1_t(L^2)}^2\Bigr),
 \end{split}
 \eeno
 which together with Lemma \ref{S4lem0} implies that
 \beno
 \begin{split}
 \|v_j\|_{L^\infty_t(L^2)}^2+\|\na v_j\|_{L^2_t(L^2)}^2\leq &Cd_j^22^{-j}\bigl(\|u_0^\h\|_{B^{\f12}}^2+\|a_0\|_{L^\infty}^2\|u_0\|_{B^{\f12}}^2\bigr)\exp\bigl(C\|u_0\|_{B^{\f12}}^2\bigr)\\
 \leq & Cd_j^22^{-j}\bigl(\|u_0^\h\|_{B^{\f12}}^2+\|a_0\|_{L^\infty}^2\bigr)\exp\bigl(C\|u_0\|_{B^{\f12}}^2\bigr).
 \end{split}
 \eeno
 This proves \eqref{S4eq5}.
 \end{proof}

\begin{lem}\label{S4lem2}
{\sl Under the assumptions of Lemma \ref{S4lem1},  we denote
\beq \label{S4eq7}
T_2^\star\eqdefa \bigl\{\ t< T^\ast:\ \|v\|_{L^2_t(L^\infty)}\leq 1.\quad\bigr\}.
\eeq Then for $t<T^\star_2$, one has
\beq\label{S4eq8}
\begin{split}
\|\na v_j\|_{L^\infty_t(L^2)}^2+&\bigl\|(\p_tv_j,\na^2v_j,\na\pi_j)\bigr\|_{L^2_t(L^2)}^2
\leq C d_j^22^j\bigl(\|u_0^\h\|_{B^{\f12}}^2+\|a_0\|_{L^\infty}^2\bigr)\exp\bigl(C\|u_0\|_{B^{\f12}}^2\bigr).
\end{split}
\eeq
}
 \end{lem}

\begin{proof}

 By taking $L^2$ inner product of the momentum equation of \eqref{S4eq3} with $\p_tv_j,$ we obtain
\beq \label{S4eq6}
\begin{split}
\f12\f{d}{dt}&\|\na v_j(t)\|_{L^2}^2+\|\sqrt{\r}\p_tv_j\|_{L^2}^2\\
=&-\int_{\R^3}\r \bigl(v+e^{t\D}u_0\bigr)\cdot\na v_j | \p_tv_j\,dx-\int_{\R^3}\r v_j\cdot\na e^{t\D}u_0 | \p_tv_j\,dx\\
&+\int_{\R^3}\r a \D e^{t\D}\D_j u_0 | \p_tv_j\,dx-\int_{\R^3}\r e^{t\D}u_0\cdot\na e^{t\D}\D_ju_0^\h | \p_tv_j^\h\,dx\\
&-\int_{\R^3}\r\bigl(e^{t\D}\D_ju_0^\h\cdot\na_\h e^{t\D}u_0^3-e^{t\D}u_0^3e^{t\D}\D_j\dive_\h u_0^\h\bigr)
 | \p_tv_j^3\,dx.
 \end{split}
 \eeq

We now deal with term by term above.

\noindent\underline{$\bullet$ \it  Estimate for $\int_{\R^3}\r \bigl(v+e^{t\D}u_0\bigr)\cdot\na v_j | \p_tv_j\,dx.$ }

\beno
\begin{split}
\bigl|\int_{\R^3}\r \bigl(v+e^{t\D}u_0\bigr)\cdot\na v_j | \p_tv_j\,dx\bigr|
\leq &C\bigl(\|v\|_{L^\infty}+\bigl\|e^{t\D}u_0\|_{L^\infty}\bigr)\|\na v_j\|_{L^2}\|\r\p_tv_j\|_{L^2}\\
\leq &C\bigl(\|v\|_{L^\infty}^2+\bigl\|e^{t\D}u_0\|_{B^{\f32}}^2\bigr)\|\na v_j\|_{L^2}^2+\f1{10}\|\sqrt{\r}\p_tv_j\|_{L^2}^2.
\end{split}
\eeno

\noindent\underline{$\bullet$ \it  Estimate for $\int_{\R^3}\r v_j\cdot\na e^{t\D}u_0 | \p_t v_j\,dx.$ }

It follows from the law of product, \eqref{S3eq0}, that
\beno
\begin{split}
\bigl|\int_{\R^3}\r v_j\cdot\na e^{t\D}u_0 | \p_t v_j\,dx\bigr|\leq &\|\na v_j\|_{L^2}\|\na e^{t\D}u_0\|_{B^{\f12}}\|\sqrt{\r}\p_tv_j\|_{L^2}\\
\leq &C\|e^{t\D}u_0\|_{B^{\f32}}^2\|\na v_j\|_{L^2}^2+\f1{10}\|\sqrt{\r}\p_tv_j\|_{L^2}^2.
\end{split}
\eeno

\noindent\underline{$\bullet$ \it  Estimate for $\int_{\R^3}\r a \D e^{t\D}\D_j u_0 | \p_tv_j\,dx.$ }

\beno
\begin{split}
\bigl|\int_{\R^3}\r a \D e^{t\D}\D_j u_0 | \p_tv_j\,dx|\leq & C\|a\|_{L^\infty}\|\D e^{t\D}\D_j u_0\|_{L^2}\|\sqrt{\r}\p_tv_j\|_{L^2}\\
\leq &C\|a_0\|_{L^\infty}^2\|\D e^{t\D}\D_j u_0\|_{L^2}^2+\f1{10}\|\sqrt{\r}\p_tv_j\|_{L^2}.
\end{split}
\eeno

\noindent\underline{$\bullet$ \it  Estimate for $\int_{\R^3}\r e^{t\D}u_0\cdot\na e^{t\D}\D_ju_0^\h | \p_tv_j^\h\,dx.$ }

\beno
\begin{split}
\bigl|\int_{\R^3}\r e^{t\D}u_0\cdot\na e^{t\D}\D_ju_0^\h | \p_tv_j^\h\,dx\big|\leq & C\|e^{t\D}u_0\|_{L^\infty}\|\na e^{t\D}\D_ju_0^\h\|_{L^2}\|\sqrt{\r}\p_tv_j\|_{L^2}\\
\leq &C\|e^{t\D}u_0\|_{B^{\f32}}^2\|\na e^{t\D}\D_ju_0^\h\|_{L^2}^2+\f1{10}\|\sqrt{\r}\p_tv_j\|_{L^2}^2.
\end{split}
\eeno

\noindent\underline{$\bullet$ \it  Estimate for $\int_{\R^3}\r\bigl(e^{t\D}\D_j u_0^\h\cdot\na_\h e^{t\D}u_0^3-e^{t\D}u_0^3e^{t\D}\D_j\dive_\h u_0^\h\bigr)
 | \p_tv_j^3\,dx.$ }

 It follows from the law of product, \eqref{S3eq0}, that
 \beno
\begin{split}
\bigl|&\int_{\R^3}\r\bigl(e^{t\D}\D_ju_0^\h\cdot\na_\h e^{t\D}u_0^3-e^{t\D}u_0^3e^{t\D}\D_j\dive_\h u_0^\h\bigr)
 | \p_tv_j^3\,dx\bigr|\\
 &\leq \bigl(\bigl\|\na_\h e^{t\D} u_0^3\bigr\|_{B^{\f12}}+\bigl\|e^{t\D}u_0^3\|_{L^\infty}\bigr)\bigl\|\na e^{t\D}\D_j u_0^\h\bigr\|_{L^2}\|\sqrt{\r}\p_tv_j\|_{L^2}\\
 &\leq C\bigl\|e^{t\D}u_0^3\|_{B^{\f32}}^2\bigl\|\na e^{t\D}\D_j u_0^\h\bigr\|_{L^2}+\f1{10}\|\sqrt{\r}\p_tv_j\|_{L^2}.
 \end{split}
 \eeno

 Inserting the above estimates into \eqref{S4eq6} leads to
 \beq \label{S4eq11}
 \begin{split}
\f{d}{dt}\|\na v_j(t)\|_{L^2}^2&+\|\sqrt{\r}\p_tv_j\|_{L^2}^2\leq C\bigl(\|v\|_{L^\infty}^2+\bigl\|e^{t\D}u_0\|_{B^{\f32}}^2\bigr)\|\na v_j\|_{L^2}^2\\
&+C\Bigl(\|a_0\|_{L^\infty}^2\|\D e^{t\D}\D_j u_0\|_{L^2}^2+\|e^{t\D}u_0\|_{B^{\f32}}^2\|\na e^{t\D}\D_ju_0^\h\|_{L^2}^2\Bigr).
\end{split}
\eeq
Applying Gronwall's inequality yields
\beno
\begin{split}
\|\na v_j\|_{L^\infty_t(L^2)}^2&+\|\sqrt{\r}\p_tv_j\|_{L^2_t(L^2)}^2\leq C\exp\bigl(\|v\|_{L^2_t(L^\infty)}^2+\bigl\|e^{t\D}u_0\|_{L^2_t(B^{\f32)}}^2\bigr)\\
&\times \Bigl(\|a_0\|_{L^\infty}^2\|\D e^{t\D}\D_j u_0\|_{L^2_t(L^2)}^2+\|e^{t\D}u_0\|_{L^2_t(B^{\f32})}^2\|\na e^{t\D}\D_ju_0^\h\|_{L^\infty_t(L^2)}^2\Bigr).
\end{split}
\eeno
Note from Lemma \ref{S4lem0} that
\beno
\begin{split}
\bigl\|e^{t\D}u_0\|_{L^2_t(B^{\f32)}}\leq &C\|u_0\|_{B^{\f12}},\quad \|\D e^{t\D}\D_j u_0\|_{L^2_t(L^2)}^2\leq Cd_j2^{\f{j}2}\|u_0\|_{B^{\f12}} \andf\\
&\|\na e^{t\D}\D_ju_0^\h\|_{L^\infty_t(L^2)}\leq Cd_j2^{\f{j}2}\|u_0\|_{B^{\f12}}.
\end{split}
\eeno
Then for $t\leq T_2^\star,$ we find
\beq\label{S4eq9}
\begin{split}
\|\na v_j\|_{L^\infty_t(L^2)}^2+\|\sqrt{\r}\p_tv_j\|_{L^2_t(L^2)}^2\leq & Cd_j^22^j\exp\bigl(C\|u_0\|_{B^{\f12}}^2\bigr) \bigl(\|a_0\|_{L^\infty}^2+\|u_0^\h\|_{B^{\f12}}^2\bigr)\| u_0\|_{B^{\f12}}^2\\
\leq & Cd_j^22^j\exp\bigl(C\|u_0\|_{B^{\f12}}^2\bigr) \bigl(\|a_0\|_{L^\infty}^2+\|u_0^\h\|_{B^{\f12}}^2\bigr).
\end{split}
\eeq
On the other hand, we deduce from \eqref{S4eq3} that
\beq \label{S4eq9a}
\begin{split}
\|(\na^2v_j,\na\pi_j)\|_{L^2}\leq &C\Bigl(\bigl(\|v\|_{L^\infty}+\|e^{t\D}u_0\|_{B^{\f32}}\bigr)
   \|\na v_j\|_{L^2}+\|\sqrt{\r}\p_tv_j\|_{L^2}\\
   &\quad\ +\|e^{t\D}u_0\|_{L^\infty}\|\na e^{t\D}\D_ju_0^\h\|_{L^2}+\|a_0\|_{L^\infty}\|\D e^{t\D}\D_j u_0\|_{L^2}\Bigr),
   \end{split}
   \eeq
Taking $L^2$ norm with respect to time yields
\beno
\begin{split}
\|(\na^2v_j,\na\pi_j)\|_{L^2_t(L^2)}\leq &C\Bigl(\bigl(\|v\|_{L^2_t(L^\infty)}+\bigl\|e^{t\D}u_0\bigr\|_{L^2_t(B^{\f32}))}\bigr)
\|\na v_j\|_{L^\infty_t(L^2)}\\
&\quad\ +\bigl\|e^{t\D}u_0\bigr\|_{L^2_t(L^\infty)}\|\na e^{t\D}\D_ju_0^\h\|_{L^\infty_t(L^2)}\\
&\quad \ +\|\sqrt{\r}\p_tv_j\|_{L^2_t(L^2)}+\|a_0\|_{L^\infty}\|\D e^{t\D}\D_j u_0\|_{L^2_t(L^2)}\Bigr),
\end{split}
\eeno which together with \eqref{S4eq9} ensures that
\beno
\|(\na^2v_j,\na\pi_j)\|_{L^2_t(L^2)}\leq Cd_j2^{\f{j}2} \bigl(\|a_0\|_{L^\infty}+\|u_0^\h\|_{B^{\f12}}\bigr)\exp\bigl(C\|u_0\|_{B^{\f12}}^2\bigr).
\eeno
Along with \eqref{S4eq9}, we conclude the proof of \eqref{S4eq8}.
\end{proof}

\begin{proof}[Proof of Proposition \ref{S4prop1}]
With \eqref{S4eq5} and \eqref{S4eq8}, we get, by a similar derivation of \eqref{S3eq7}, that
\beno
\|v\|_{\wt{L}^\infty(B^{\f12})}+\|\na v\|_{\wt{L}^2_t(B^{\f12})}\leq \bigl(\|a_0\|_{L^\infty}^2+\|u_0^\h\|_{B^{\f12}}^2\bigr)\exp \bigl(C\|u_0\|_{B^{\f12}}^2\bigr)\quad
\mbox{for}\quad t\leq T^\star_2, \eeno
which together with \eqref{small1} ensures that \eqref{S4eq10} holds for $t\leq T_2^\star.$

Then for $t\leq T_2^\star,$ we have
\beno
\|v\|_{L^2_t(L^\infty)}\leq C\|v\|_{\wt{L}^2_t(B^{\f32})}\leq C\e_0,
\eeno
for $\e_0$ given by \eqref{small1}. In particular, if we take $\e_0$ in \eqref{small1} so small that $C\e_0\leq\f12,$ we deduce by a continuous argument  that
$T_2^\star$ determined by \eqref{S4eq7} can be any number smaller than $T^\ast.$ This proves \eqref{S4eq10}.
\end{proof}

\begin{rmk}
Once again by virtue of \eqref{S4eq10}, it follows from classical theory of inhomogeneous incompressible Navier-Stokes
system that $T^\ast=\infty.$
\end{rmk}

\begin{col}\label{S4col1}
{\sl Under the assumptions of Proposition \ref{S4prop1},  for any $t>0$ and $\eta$ given by \eqref{small1}, we have
\beq \label{S4eq12}
\begin{split}
 \|\sqrt{t}\na v_j\|_{L^\infty_t(L^2)}^2+&\|\sqrt{t}(\p_t v_j,\na^2 v_j, \na\pi_j)\|_{L^2_t(L^2)}^2
 \leq
  C\eta d_j^22^{-j}.
  \end{split}
 \eeq }
\end{col}

\begin{proof} We first get, by multiplying $t$ to \eqref{S4eq11} and then applying Gronwall's inequality, that
\beno
 \begin{split}
\|\sqrt{t}\na v_j\|_{L^\infty_t(L^2)}^2+\|\sqrt{\r t}\p_tv_j\|_{L^2_t(L^2)}^2\leq & C\exp \bigl(C\|v\|_{L^2_t(L^\infty)}^2+C\bigl\|e^{t\D}u_0\|_{L^2_t(B^{\f32})}^2\bigr)\\
&\ \times\Bigl(\|\na v_j\|_{L^2_t(L^2)}^2+\|a_0\|_{L^\infty}^2\|\sqrt{t}\D e^{t\D}\D_j u_0\|_{L^2_t(L^2)}^2\\
&\qquad\quad +\|e^{t\D}u_0\|_{L^2_t(B^{\f32})}^2\|\sqrt{t}\na e^{t\D}\D_ju_0^\h\|_{L^\infty_t(L^2)}^2\Bigr).
\end{split}
\eeno
Yet it follows Lemma \ref{S4lem0} that
\beno
\begin{split}
\|\sqrt{t}\D e^{t\D}\D_j u_0\|_{L^2_t(L^2)}^2\leq &C2^{4j}\int_0^t t'e^{-ct'2^{2j}}\|\D_ju_0\|_{L^2}^2\,dt\\
\leq &C2^{2j}\int_0^t e^{-\f{c}2t'2^{2j}}\,dt\|\D_ju_0\|_{L^2}^2\\
\leq &C\|\D_ju_0\|_{L^2}^2\leq  Cd_j^22^{-j}\|u_0\|_{B^{\f12}}^2.
\end{split}
\eeno
and
\beno
\begin{split}
\|\sqrt{t}\na e^{t\D}\D_ju_0^\h(t)\|_{L^2}^2\leq & Ct2^{2j}e^{-ct'2^{2j}}\|\D_ju_0^\h\|_{L^2}^2\\
\leq & C\|\D_ju_0^\h\|_{L^2}^2\leq  Cd_j^22^{-j}\|u_0^\h\|_{B^{\f12}}^2,
\end{split}
\eeno
As a result, we deduce from Proposition \ref{S4prop1} that
\beno
\|\sqrt{t}\na v_j\|_{L^\infty_t(L^2)}^2+\|\sqrt{ t}\p_tv_j\|_{L^2_t(L^2)}^2\leq C\eta d_j^22^{-j}.
\eeno
This together with \eqref{S4eq9a} ensures that
\beno
\|\sqrt{t}(\na^2 v_j, \na\pi_j)\|_{L^2_t(L^2)}^2
 \leq
  C\eta d_j^22^{-j}.
  \eeno
This proves \eqref{S4eq12}.
\end{proof}

\begin{lem}\label{S4lem3}
{\sl Under the assumptions of proposition \ref{S4prop1}, for any $t>0,$  we have
\beq \label{S4eq16}
\begin{split}
\f{d}{dt}\|&\sqrt{\r}\p_tv_j\|_{L^2}^2+\|\na\p_tv_j\|_{L^2}^2 \leq  \|\na\p_t e^{t\D}\D_ju_0^\h\|_{L^2}^2+C\|\sqrt{\r}\p_tv_j\|_{L^2}^2
\\
&\times \bigl(\|u\|_{B^{\f32}}^2+
\|t^{\f14}u_t\|_{L^2}^2+\|e^{t\D}u_0\|_{B^{\f32}}^2+\|\na e^{t\D}u_0\|_{B^{\f32}}+\|\sqrt{t}e^{t\D}u_0\|_{B^{\f72}} \bigr)\\
&+C\bigl[\bigl(t^{-1}+t^{-\f12}\bigl(\|v\|_{L^\infty}+
 \|e^{t\D}u_0\|_{B^{\f32}}\bigr)\|t^{\f14}u_t\|_{L^2}^2 +\|e^{t\D}u_0\|_{B^{\f52}}^2\\
&\qquad +t^{-\f12}
\|e^{t\D}u_0\|_{B^{\f72}}+\|u\|_{L^\infty}^2\bigl(\|u\|_{L^\infty}^2+\|e^{t\D}u_0\|_{B^{\f32}}^2\bigr)\bigr]\|\na v_j\|_{L^2}^2 \\
&+ Ct^{-\f12}\|t^{\f14}u_t\|_{L^2}^2\bigl(
 \|e^{t\D}u_0\|_{B^{\f32}}\|\na e^{t\D}\D_ju_0^\h\|_{L^2} +\|a_0\|_{L^\infty} \|\na^2 e^{t\D}\D_ju_0^\h\|_{L^2} \bigr)\|\na v_j\|_{L^2} \\
&+ \|a_0\|_{L^\infty}^2\bigl(\|\na\D e^{t\D}\D_ju_0\|_{L^2}^2   +\|u\|_{L^\infty}^2\|\D e^{t\D}\D_ju_0\|_{L^2}^2\bigr)\\
&+\bigl( \|e^{t\D}u_0\|_{B^{\f52}}^2+\|u\|_{L^\infty}^2\|e^{t\D}u_0\|_{B^{\f32}}^2+t^{-\f12}\|\p_te^{t\D}u_0^3\|_{B^{\f32}} \bigr) \|\na e^{t\D}\D_ju_0^\h\|_{L^2}^2\\
&+\bigl(t^{-1}+\|e^{t\D}u_0\|_{B^{\f32}}^2\bigr)\|\na^2 e^{t\D}\D_ju_0^\h\|_{L^2}^2+\|a_0\|_{L^\infty}\|\D\p_te^{t\D}\D_ju_0\|_{L^2}\|\sqrt{\r}\p_tv_j\|_{L^2}.
\end{split}
\eeq}
\end{lem}

\begin{proof} In the rest of this section, we shall always denote $u\eqdefa v+e^{t\D}u_0$ and $D_t=\p_t+u\cdot\na.$ Then we get, by applying $\p_t$ to
 the $v_j$ equation
of \eqref{S4eq3}, that
\beq \label{S4eq13}
\begin{split}
\r\p_t^2v_j+\r u\cdot\na \p_tv_j+&\p_t\bigl(\r v_j\cdot\na e^{t\D} u_0\bigr)-\Delta \p_tv_j+\na \p_t\pi_j\\
 =&-\r_t D_tv_j-\r\p_tu\cdot\na\p_tv_j+\p_t(\r F_j)+\p_t\bigl(\r {a}\D e^{t\D}\D_j u_0\bigr).
\end{split}
\eeq
Taking $L^2$ inner product of the above equation with $\p_tv_j,$ we obtain
\beq \label{S4eq14}
\begin{split}
\f12\f{d}{dt}\|\sqrt{\r}\p_tv_j(t)\|_{L^2}^2&+\|\na\p_tv_j\|_{L^2}^2=-\int_{\R^3}\bigl(\r_t D_tv_j+\r\p_tu\cdot\na\p_tv_j\bigr) | \p_tv_j\,dx\\
&-\int_{\R^3}\p_t\bigl(\r v_j\cdot\na e^{t\D} u_0\bigr)   | \p_tv_j\,dx  +\int_{\R^3}\p_t(\r F_j)  | \p_tv_j\,dx\\
&+\int_{\R^3}\p_t\bigl(\r {a}\D e^{t\D}\D_j u_0\bigr)   | \p_tv_j\,dx.
\end{split}
\eeq

\noindent\underline{$\bullet$ \it  Estimate for $\int_{\R^3}\bigl(\r_t D_tv_j+\r\p_tu\cdot\na\p_tv_j\bigr) | \p_tv_j\,dx.$ }

Observing that
\beno
\begin{split}
\bigl|\int_{\R^3} u\cdot\na v_j | \r u\cdot\na \p_tv_j\,dx\bigr|\leq &C\|u\|_{L^\infty}^2\|\na v_j\|_{L^2}\|\na\p_tv_j\|_{L^2}\\
\leq &C\|u\|_{L^\infty}^4\|\na v_j\|_{L^2}^2+\f1{24}\|\na\p_tv_j\|_{L^2}^2.
\end{split}
\eeno
Then along the same line to the proof of \eqref{S3eq17}, we have
\beno\begin{split}
\bigl|\int_{\R^3}\bigl(\r_t D_tv_j+\r\p_tu\cdot\na\p_tv_j\bigr) |& \p_tv_j\,dx\bigr|
\leq  C\Bigl(\|u\|_{B^{\f32}}^2\|\sqrt{\r}\p_tv_j\|_{L^2}^2+\|u\|_{L^\infty}^4\|\na v_j\|_{L^2}^2\\
&+\|u_t\|_{L^2}^2\|\na v_j\|_{L^2}\bigl(\|\na^2v_j\|_{L^2}+\|\sqrt{\r}\p_tv_j\|_{L^2}\bigr)\Bigr)+\f1{8}\|\na\p_tv_j\|_{L^2}^2,
\end{split}
\eeno
which together with \eqref{S4eq9a} ensures that
\beno\begin{split}
\bigl|\int_{\R^3}\bigl(\r_t D_tv_j&+\r\p_tu\cdot\na\p_tv_j\bigr) | \p_tv_j\,dx\bigr|
\leq  C\Bigl(\bigl(\|u\|_{B^{\f32}}^2+\|t^{\f14}u_t\|_{L^2}^2\bigr)\|\sqrt{\r}\p_tv_j\|_{L^2}^2\\
&+\bigl[\|t^{\f14}u_t\|_{L^2}^2\bigl(t^{-1}+t^{-\f12}(\|v\|_{L^\infty}+\|e^{t\D}u_0\|_{B^{\f32}})\bigr)+\|u\|_{L^\infty}^4\bigr]\|\na v_j\|_{L^2}^2\\
&+\bigl(\|e^{t\D}u_0\|_{B^{\f32}}\|\na e^{t\D}\D_ju_0^\h\|_{L^2}+\|a_0\|_{L^\infty}\|\D e^{t\D}\D_ju_0\|_{L^2}\bigr)\\
&\qquad\qquad\qquad\qquad\qquad\times t^{-\f12}\|t^{\f14}u_t\|_{L^2}^2\|\na v_j\|_{L^2}\Bigr)
+  \f1{8}\|\na\p_tv_j\|_{L^2}^2.
\end{split}
\eeno

\noindent\underline{$\bullet$ \it  Estimate for $\int_{\R^3}\p_t\bigl(\r v_j\cdot\na e^{t\D} u_0\bigr)   | \p_tv_j\,dx.$ }

It is easy to observe that
\beno
\begin{split}
\int_{\R^3}\p_t\bigl(\r v_j\cdot\na e^{t\D} u_0\bigr)  & | \p_tv_j\,dx=\int_{\R^3}\p_t\r v_j\cdot\na e^{t\D} u_0   | \p_tv_j\,dx\\
&+\int_{\R^3}\r \p_t v_j\cdot\na e^{t\D} u_0  | \p_tv_j\,dx+\int_{\R^3}\r  v_j\cdot\na \p_t e^{t\D} u_0  | \p_tv_j\,dx.
\end{split}
\eeno
Thanks to the transport equation of \eqref{S1eq1}, we get, by using integration by parts, that
\beno
\begin{split}
\int_{\R^3}\p_t\r v_j\cdot\na e^{t\D} u_0 &  | \p_tv_j\,dx=\int_{\R^3}\r (u\cdot\na v_j)\cdot\na e^{t\D} u_0   | \p_tv_j\,dx\\
&+\int_{\R^3}\r v_j\otimes u: \na^2e^{t\D}u_0   | \p_tv_j\,dx+\int_{\R^3}\r v_j\cdot\na  e^{t\D} u_0   | u\cdot\na\p_tv_j\,dx.
\end{split}
\eeno
It is easy to observe that
\beno
\begin{split}
\bigl|\int_{\R^3}\r (u\cdot\na v_j)\cdot\na e^{t\D} u_0   | \p_tv_j\,dx\bigr|\leq &C\|u\|_{L^\infty}\|\na v_j\|_{L^2}\|\na e^{t\D}u_0\|_{L^\infty}\|\sqrt{\r}\p_tv_j\|_{L^2}\\
\leq &C\bigl(\|u\|_{L^\infty}^2\|\sqrt{\r}\p_tv_j\|_{L^2}^2+\|\na e^{t\D}u_0\|_{B^{\f32}}^2\|\na v_j\|_{L^2}^2\bigr).
\end{split}
\eeno
Similarly, one has
\beno
\begin{split}
\bigl|\int_{\R^3}\r v_j\otimes u: \na^2e^{t\D}u_0   | \p_tv_j\,dx\bigr|\leq & C\|v_j\|_{L^6}\|u\|_{L^\infty}\|\na^2e^{t\D}u_0\|_{L^3}\|\sqrt{\r}\p_tv_j\|_{L^2}\\
\leq &C\bigl(\|u\|_{L^\infty}^2\|\sqrt{\r}\p_tv_j\|_{L^2}^2+\|\na e^{t\D}u_0\|_{B^{\f32}}^2\|\na v_j\|_{L^2}^2\bigr),
\end{split}
\eeno
and
\beno
\begin{split}
\bigl|\int_{\R^3}\r v_j\cdot\na  e^{t\D} u_0   | u\cdot\na\p_tv_j\,dx\bigr|\leq &C\|v_j\|_{L^6}\|\na e^{t\D}u_0\|_{L^3}\|u\|_{L^\infty}\|\na\p_tv_j\|_{L^2}\\
\leq &C\|u\|_{L^\infty}^2\|e^{t\D}u_0\|_{B^{\f32}}^2\|\na v_j\|_{L^2}^2+\f1{8}\|\na\p_tv_j\|_{L^2}^2.
\end{split}
\eeno
Whereas we notice that
\beno
\begin{split}
\bigl|\int_{\R^3}\r \p_tv_j\cdot\na e^{t\D}u_0 | \p_tv_j\,dx\bigr|\leq &\bigl\| \na e^{t\D}u_0\|_{L^\infty}\|\sqrt{\r} \p_tv_j\|_{L^2}^2,
\end{split}
\eeno
and   it follows from the law of product, \eqref{S3eq0}, that
\beno
\begin{split}
\bigl|\int_{\R^3}\r  v_j\cdot\na \p_t e^{t\D} u_0  | \p_tv_j\,dx  \bigr|\leq &C\|\na v_j\|_{L^2}\|\na  \p_t e^{t\D} u_0\|_{B^{\f12}}\|\sqrt{\r} \p_tv_j\|_{L^2}\\
\leq & Ct^{-\f12} \|\p_t e^{t\D} u_0\|_{B^{\f32}} \|\na v_j\|_{L^2}^2+ t^{\f12} \|\p_t e^{t\D} u_0\|_{B^{\f32}} \|\sqrt{\r} \p_tv_j\|_{L^2}^2.
\end{split}
\eeno
As a result, it comes out
\beno
\begin{split}
\bigl|\int_{\R^3}\p_t\bigl(\r v_j&\cdot\na e^{t\D} u_0\bigr)   | \p_tv_j\,dx  \bigr| \leq   \f1{8}\|\na\p_tv_j\|_{L^2}^2 \\
&\quad+ C\Bigl(\bigl(\|u\|_{L^\infty}^2+
 \bigl\|\na e^{t\D}u_0\|_{B^{\f32}}   +   t^{\f12} \|\p_t e^{t\D} u_0\|_{B^{\f32}} \bigr)\|\sqrt{\r} \p_tv_j\|_{L^2}^2 \\
&\quad+\bigl(\|e^{t\D}u_0\|_{B^{\f52}}^2+\|  u\|_{L^\infty}^2\|e^{t\D}u_0\|_{B^{\f32}}^2+t^{-\f12}
\|\p_t e^{t\D} u_0\|_{B^{\f32}}\bigr)\|\na v_j\|_{L^2}^2 \Bigr).
\end{split}
\eeno
\noindent\underline{$\bullet$ \it  Estimate for $\int_{\R^3}\p_t\bigl(\r {a}\D e^{t\D}\D_j u_0\bigr)   | \p_tv_j\,dx.$ }

Again thanks to the transport equation of \eqref{S4eq1} and notice that $\r=\f1{1+a},$ we get, by using integration by parts, that
\beno
\begin{split}
\int_{\R^3}&\p_t\bigl(\r {a}\D e^{t\D}\D_j u_0\bigr)   | \p_tv_j\,dx=\int_{\R^3}\f{a}{1+a} u\cdot\na\D e^{t\D}\D_j u_0   | \p_tv_j\,dx\\
&+\int_{\R^3}\f{a}{1+a}\D e^{t\D}\D_j u_0   | u\cdot\na\p_tv_j\,dx+\int_{\R^3}\f{a}{1+a}\D \p_t e^{t\D}\D_j u_0  | \p_tv_j\,dx,
\end{split}
\eeno
from which, we infer
\beno
\begin{split}
\bigl|\int_{\R^3}\p_t\bigl(\r {a}\D e^{t\D}\D_j u_0\bigr)&   | \p_tv_j\,dx\bigr|
\leq C\|a_0\|_{L^\infty}\Bigl(\|u\|_{L^\infty}\bigl(\|\sqrt{\r}\p_tv_j\|_{L^2}\bigl\|\na\D e^{t\D}\D_j u_0 \bigr\|_{L^2}\\
&+\|\na\p_tv_j\|_{L^2}\bigl\|\D e^{t\D}\D_j u_0 \bigr\|_{L^2}\bigr)+\bigl\|\D \p_t e^{t\D}\D_j u_0\bigr\|_{L^2}\|\sqrt{\r}\p_tv_j\|_{L^2}\Bigr).
\end{split}
\eeno
Then applying Young's inequality gives
\beno
\begin{split}
\bigl|\int_{\R^3}\p_t\bigl(\r{a}\D e^{t\D}\D_j u_0\bigr)   |& \p_tv_j\,dx\bigr|
\leq \f1{8}\|\na\p_tv_j\|_{L^2}^2+C\Bigl(\|u\|_{L^\infty}^2\|\sqrt{\r}\p_tv_j\|_{L^2}^2\\
&+\|a_0\|_{L^\infty}\bigl\|\D \p_t e^{t\D}\D_j u_0\bigr\|_{L^2}\|\sqrt{\r}\p_tv_j\|_{L^2}\\
&+\|a_0\|_{L^\infty}^2\bigl(\bigl\|\na\D e^{t\D}\D_j u_0 \bigr\|_{L^2}^2
+\|u\|_{L^\infty}^2\bigl\|\D e^{t\D}\D_j u_0 \bigr\|_{L^2}^2\bigr)\Bigr).
\end{split}
\eeno

\noindent\underline{$\bullet$ \it  Estimate for $\int_{\R^3}\p_t\r F_j  | \p_tv_j\,dx.$ }

In view of the transport equation of \eqref{S1eq1}, we get, by using integration by parts, that
\beno
  \int_{\R^3}\p_t\r F_j  | \p_tv_j\,dx=\int_{\R^3}\r u\cdot\na F_j   | \p_tv_j\,dx+ \int_{\R^3} F_j   | \r u\cdot\na\p_tv_j\,dx,
  \eeno
  from which, we infer
\beno
\begin{split}
\bigl|\int_{\R^3}\p_t\r F_j  | \p_tv_j\,dx \bigr|\leq & C\|u\|_{L^\infty}\bigl(\|\sqrt{\r}\p_tv_j\|_{L^2}\|\na F_j\|_{L^2}
+ \|F_j\|_{L^2}\|\na\p_tv_j\|_{L^2}\bigr) \\
\leq &C\|u\|_{L^\infty}^2\|\sqrt{\r}\p_tv_j\|_{L^2}^2+\|\na F_j\|_{L^2}^2+\|u\|_{L^\infty}^2\|F_j\|_{L^2}^2+\f1{8}\|\na\p_tv_j\|_{L^2}^2.
\end{split}
\eeno
Yet it follows from the law of product in Besove spaces that
\beno
\|F_j\|_{L^2}\leq C\|e^{t\D}u_0\|_{B^{\f32}}\|\na e^{t\D}\D_ju_0^\h\|_{L^2},
\eeno
and
 \beno
 \|\na F_j\|_{L^2}\leq C\bigl(\|e^{t\D}u_0\|_{B^{\f32}}\|\na^2 e^{t\D}\D_ju_0^\h\|_{L^2}+\|e^{t\D}u_0\|_{B^{\f52}}\|\na e^{t\D}\D_ju_0^\h\|_{L^2}\bigr).
 \eeno
Hence we obtain
\beno
\begin{split}
\bigl|\int_{\R^3}\p_t\r F_j & | \p_tv_j\,dx \bigr|
\leq C\Bigl(\|u\|_{L^\infty}^2\|\sqrt{\r}\p_tv_j\|_{L^2}^2+\|e^{t\D}u_0\|_{B^{\f32}}^2\|\na^2 e^{t\D}\D_ju_0^\h\|_{L^2}^2 \\
&    +\|e^{t\D}u_0\|_{B^{\f52}}^2\|\na e^{t\D}\D_ju_0^\h\|_{L^2}^2
+\|u\|_{L^\infty}^2\|e^{t\D}u_0\|_{B^{\f32}}^2\|\na e^{t\D}\D_ju_0^\h\|_{L^2}^2\Bigr)+\f1{8}\|\na\p_tv_j\|_{L^2}^2.
\end{split}
\eeno

\noindent\underline{$\bullet$ \it  Estimate for $\int_{\R^3}\r \p_t F_j  | \p_tv_j\,dx.$ }

It follows from the law of product in Sobolev spaces \eqref{S3eq0} that
\beno
\begin{split}
\|\p_tF_j\|_{L^2}\leq &\|\p_te^{t\D}u_0\|_{B^{\f12}}\|\na^2 e^{t\D}\D_ju_0^\h\|_{L^2}+ \|e^{t\D}u_0\|_{B^{\f32}}\|\na \p_te^{t\D}\D_ju_0^\h\|_{L^2}  \\
&+\|\p_te^{t\D}\D_j u_0^\h\|_{L^2}\|\na_\h e^{t\D}u_0^3\|_{L^\infty}+ \|\na e^{t\D}\D_j u_0^\h\|_{L^2}\|\na_\h \p_te^{t\D}u_0^3\|_{B^{\f12}}.
\end{split}
\eeno
This implies that
\beno
  \begin{split}
\bigl|\int_{\R^3}\r \p_t F_j  | \p_tv_j\,dx\bigr|\leq & C\bigl(\|t^{\f12}\p_te^{t\D}u_0\|_{B^{\f12}}^2+\|e^{t\D}u_0\|_{B^{\f32}}^2+  \|t^{\f12}\na_\h e^{t\D}u_0^3\|_{L^\infty}^2\\
  &+\|t^{\f12}\p_te^{t\D}u_0^3\|_{B^{\f32}} \bigr)\|\sqrt{\r}\p_tv_j\|_{L^2}^2
   +t^{-1}\|\na^2e^{t\D}\D_j u_0^\h\|_{L^2}^2\\
   &+ \|\na\p_t e^{t\D}\D_j u_0^\h\|_{L^2}^2
  +t^{-\f12}\| \p_te^{t\D}u_0^3\|_{B^{\f32}} \|\na e^{t\D}\D_j u_0^\h\|_{L^2}^2.
  \end{split}  \eeno
 Inserting the above estimates into \eqref{S4eq14} leads to \eqref{S4eq16}. This completes the proof of the Lemma.
 \end{proof}

 Before proceeding, we also need the following lemma:

 \begin{lem}\label{S4lem4}
 {\sl For any $s\geq 0,$ one has
  \beq \label{S4eq17}
  \bigl\|t^s e^{t\D} u_0\bigr\|_{L^1_t(B^{2s+\f52})}+  \bigl\|t^s e^{t\D} u_0\bigr\|_{L^2_t(B^{2s+\f32})}\leq C\|u_0\|_{B^{\f12}}.
  \eeq}
  \end{lem}
  \begin{proof}     In view of Definition \ref{defbesov} and Lemma \ref{S4lem0}, we have
  \beno
  \begin{split}
   \bigl\|t^s e^{t\D} u_0\bigr\|_{L^1_t(B^{2s+\f52})}\leq &C\sum_{j\in\Z}\int_0^t(t')^s2^{j\left(2s+\f52\right)} 2^{-ct'2^{2j}}\|\D_j u_0\|_{L^2}\,dt'\\
   \leq &C\sum_{j\in\Z}2^{\f{5j}2}\int_0^t 2^{-ct'2^{2j}}\,dt' \|\D_j u_0\|_{L^2}\\
    \leq &C\sum_{j\in\Z}2^{\f{j}2}\|\D_j u_0\|_{L^2}  \leq C \|u_0\|_{B^{\f12}}.
  \end{split}
  \eeno
 While we get, by applying Minkowsky's inequality, that
 \beno
 \begin{split}
 \bigl\|t^s e^{t\D} u_0\bigr\|_{L^2_t(B^{2s+\f32})}\leq &\sum_{j\in\Z}2^{j\left(2s+\f32\right)}\Bigl(\int_0^t (t')^{2s}\|\D_j e^{t\D}u_0\|_{L^2}^2\,dt'\Bigr)^{\f12}\\
 \leq & \sum_{j\in\Z}2^{j\left(2s+\f32\right)}\Bigl(\int_0^t (t')^{2s} e^{-ct2^{2j}}\,dt'\Bigr)^{\f12}  \| \D_ju_0\|_{L^2}\\
  \leq & \sum_{j\in\Z}2^{\f{j}2} \| \D_ju_0\|_{L^2}  \leq C \|u_0\|_{B^{\f12}}.
 \end{split}
 \eeno
  This completes the proof of the lemma.
  \end{proof}

 \begin{proof}[Proof of Proposition \ref{S4prop2}]    By multiplying \eqref{S4eq16} by $t$ and then applying Gronwall's inequality, we obtain
 \beno
 \begin{split}
 \|\sqrt{t}\p_tv_j&\|_{L^\infty_t(L^2)}^2+\|\na\p_tv_j\|_{L^2_t(L^2)}^2 \leq C\exp\Bigl(C\bigl(\|u\|_{L^2_t(B^{\f32})}^2+
 \|t^{\f14}u_t\|_{L^2_t(L^2)}^2\\
 &+\|e^{t\D}u_0\|_{L^2_t(B^{\f32})}^2+\|e^{t\D}u_0\|_{L^1_t(B^{\f52})}
 +\|\sqrt{t}e^{t\D}u_0\|_{L^1_t(B^{\f72})} \bigr)\Bigr)\\
 &\times\Bigl(\|\sqrt{\r}\p_tv_j\|_{L^2_t(L^2)}^2 + \bigl\|t^{\f12}\na\p_t e^{t\D}\D_ju_0^\h\bigr\|_{L^2_t(L^2)}^2+\bigl[\bigl\|t^{\f12}e^{t\D}u_0\bigr\|_{L^2_t(B^{\f52})}^2\\
 &\ +\bigl(1+\bigl(\|t^{\f12} v\|_{L^\infty_t(B^{\f32})}+  \bigl\|t^{\f12}e^{t\D}u_0\bigr\|_{L^\infty_t(B^{\f32})}\bigr)\|t^{\f14}u_t\|_{L^2_t(L^2)}^2
 +
 \bigl\|t^{\f12}e^{t\D}u_0\bigr\|_{L^1_t(B^{\f72})}\\
 &\ +\|u\|_{L^2_t(L^\infty)}^2\bigl(\bigl\|t^{\f12} e^{t\D}u_0\bigr\|_{L^2_t(B^{\f32})}^2+\|t^{\f12}u\|_{L^\infty_t(L^\infty)}^2\bigr)\bigr]
 \|\na v_j\|_{L^\infty_t(L^2)}^2\\
 &\ + \bigl(
  \bigl\|t^{\f12}e^{t\D}u_0\bigr\|_{L^\infty_t(B^{\f32})}\|\na e^{t\D}\D_ju_0^\h\|_{L^\infty_t(L^2)}+\|a_0\|_{L^\infty}^2 \bigl\|t^{\f12}\na^2 e^{t\D}\D_ju_0^\h\bigr\|_{L^\infty_t(L^2)} \bigr)\\
 &\  \times\|t^{\f14}u_t\|_{L^2_t(L^2)}^2\|\na v_j\|_{L^\infty_t(L^2)}+\|a_0\|_{L^\infty}^2\bigl(\|u\|_{L^2_t(L^\infty)}^2\bigl\|t^{\f12}\D e^{t\D}\D_ju_0\bigr\|_{L^\infty_t(L^2)}^2\\
  &\ +\bigl\|t^{\f12}\D\p_te^{t\D}\D_ju_0\bigr\|_{L^1_t(L^2)}^2
  + \bigl\|t^{\f12}\na\D e^{t\D}\D_ju_0\|_{L^2_t(L^2)}^2 \bigr)
 +\bigl( \bigl\|t^{\f12}e^{t\D}u_0\bigr\|_{L^2_t(B^{\f52})}^2
 \\
 &\ +\bigl\|t^{\f12}\p_te^{t\D}u_0^3\bigr\|_{L^1_t(B^{\f32})}+\|u\|_{L^2_t(L^\infty)}^2\bigl\|t^{\f12}e^{t\D}u_0\bigr\|_{L^\infty_t(B^{\f32})}^2  \bigr) \|\na e^{t\D}\D_ju_0^\h\|_{L^\infty_t(L^2)}^2\\
 &\ +\bigl(1+\bigl\|t^{\f12} e^{t\D}u_0\bigr\|_{L^\infty_t(B^{\f32})}^2\bigr)\|\na^2 e^{t\D}\D_ju_0^\h\|_{L^2_t(L^2)}^2 \Bigr).
 \end{split}
 \eeno
 Let us denote
 \beq \label{S4eq19}
 T_3^\star\eqdefa \bigl\{\ t>0:\     \|t^{\f12} v\|_{L^\infty_t(B^{\f32})}\leq 1\ \bigr\}.
 \eeq
 In view of \eqref{S4eq8} and \eqref{S4eq12}, we get, by a similar derivation of \eqref{S3eq12}, that
 \beno
 \|t^{\f14}v_t\|_{L^2_t(L^2)}\leq C\eta,
 \eeno
 which together with Lemma \ref{S4lem4} ensures that
 \beq \label{S4eq19a}
 \begin{split}
 \|t^{\f14}u_t\|_{L^2_t(L^2)}\leq &\|t^{\f14} \p_te^{t\D}u_0\|_{L^2_t(L^2)}+\|t^{\f14}v_t\|_{L^2_t(L^2)}\\
 \leq &\|t^{\f14} e^{t\D}u_0\|_{L^2_t(B^2)}+\|t^{\f14}v_t\|_{L^2_t(L^2)}\leq C\bigl(\eta+\|u_0\|_{B^{\f12}}\bigr).
 \end{split}
 \eeq
 Then we deduce from Lemmas \ref{S4lem0}, \ref{S4lem2} and  \ref{S4lem4} that
\beq \label{S4eq19}
\|\sqrt{t}\p_tv_j\|_{L^\infty_t(L^2)}^2+\|\sqrt{t}\na\p_tv_j\|_{L^2_t(L^2)}^2 \leq C\eta
  d_j^22^{j} \eeq
for $t\leq T^\star_3$ and $\eta$ given by \eqref{small1}.

By virtue of \eqref{S4eq9a}, \eqref{S4eq19} and Lemma \ref{S4lem0}, we infer
\beq\label{S4eq19b}
\begin{split}
\|\sqrt{t}\na^2v_j\|_{L^\infty_t(L^2)}\leq C\Bigl(&\bigl(\|\sqrt{t}v\|_{L^\infty_t(L^\infty)}
+\|\sqrt{t}e^{t\D}u_0\|_{L^\infty_t(B^{\f32})}\bigr)
   \|\na v_j\|_{L^\infty_t(L^2)}\\
   &+\|\sqrt{t}\p_tv_j\|_{L^\infty_t(L^2)}+\|\sqrt{t}e^{t\D}u_0\|_{L^\infty_t(B^{\f32})} \|\na e^{t\D}\D_ju_0^\h\|_{L^\infty_t(L^2)}\\
   &+\|a_0\|_{L^\infty}\|\sqrt{t}\D e^{t\D}\D_j u_0\|_{L^\infty_t(L^2)}\Bigr)\\
 \leq C &
 \eta d_j^22^{j}
\end{split}
   \eeq for $t\leq T^\star_3$ and $\eta$ given by \eqref{small1}.

Combining \eqref{S4eq12} with \eqref{S4eq19} and \eqref{S4eq19b}, we achieve \eqref{S4eq18} for  $t\leq T^\star_3.$
Then under the assumption of \eqref{small1}, we have
\beno
   \|t^{\f12} v\|_{L^\infty_t(B^{\f32})}\leq \f12  \quad\mbox{for}\ \ t\leq T^\star_3
   \eeno
as long as $\e_0$ in \eqref{small1} is small enough. This contradict with the definition of $T^\star_3$ determined by \eqref{S4eq19}.
 This  in turn shows that $T^\star_3=\infty,$ and we complete the proof of Proposition \ref{S4prop2}.
 \end{proof}

 Exactly along the same line to the proof of \eqref{S4eq19} and \eqref{S4eq19b}, we have the following corollary:

 \begin{col}\label{S4col2}
 {\sl Under the assumption of proposition \ref{S4prop2}, for any $t>0,$ we have
 \beq \label{S4eq20}
 \|{t}(\p_tv_j,\na^2v_j)\|_{L^\infty_t(L^2)}^2+\|t\na\p_tv_j\|_{L^2_t(L^2)}^2 \leq C
  \eta d_j^22^{-j} \eeq  for $\eta$ given by \eqref{small1}   }
 \end{col}

Now let us turn to the proof of Proposition \ref{S4prop3}.

 \begin{proof}[Proof of Proposition \ref{S4prop3}]  The proof of this proposition basically follows from that of Proposition \ref{S3prop3}.
 By applying the operator $D_t=\p_t+u\cdot\na $  to the $v_j$ equation of \eqref{S4eq3},
 we get, by a similar derivation of \eqref{S3eq19}, that
  \beq \label{S4eq22}
 \begin{split}
 \r D_t^2v_j-\D D_tv_j&+\na D_t\pi_j=-\D u\cdot\na v_j-2\sum_{i=1}^3\p_iu\cdot\na\p_i v_j+\na u\cdot\na\pi_j\\
 &\qquad-\r D_t\bigl(v_j\cdot\na e^{t\D}u_0\bigr)+\r D_tF_j+\r{a}D_t\D e^{t\D}\D_ju_0\eqdefa G_j.
 \end{split}
 \eeq
 Then along the same line to the proof of \eqref{S3eq21}, we write
 \beq \label{S4eq23}                                                                                \begin{split}
 \f12\f{d}{dt}\|t\na D_tv_j(t)\|_{L^2}^2+\|t\sqrt{\r}&D_t^2v_j\|_{L^2}^2
  \leq \|\sqrt{t}\na D_tv_j\|_{L^2}^2+t^2\int_{\R^3} G_j |  D_t^2v_j\,dx\\
  &-t^2\int_{\R^3} \na D_tv_j | [\na; D_t]D_tv_j\,dx-t^2\int_{\R^3}\na D_t\pi_j |  D_t^2v_j\,dx.
 \end{split}
 \eeq
 We first deal with the estimate of $\|tG_j\|_{L^2_t(L^2)}.$
 Notice from \eqref{S4eq3} that
 \beno
 \left\{
 \begin{array}{ll}
 -\Delta v_j+\na \pi_j=-\r\bigl(\p_tv_j+u\cdot\na v_j+v_j\cdot\na e^{t\D} u_0-F_j- a\D e^{t\D}\D_j u_0\bigr),\\
 \dive v_j=0,
 \end{array}
 \right.
 \eeno
 from which, we deduce from the classical theory on Stokes operator that
 \beno
 \begin{split}
 \|\sqrt{t}(\na^2v_j,\na\pi_j)\|_{L^2_t(L^6)}\leq C\Bigl(& \|\sqrt{t}\na\p_tv_j\|_{L^2_t(L^2)}
 +\|u\|_{L^2_t(L^\infty)}\|\sqrt{t}\na^2v_j\|_{L^\infty_t(L^2)} \\
 &+\|\na v_j\|_{L^\infty_t(L^2)}\|\sqrt{t}e^{t\D}u_0\|_{L^2_t(B^{\f52})}+\|\sqrt{t}F_j\|_{L^2_t(L^6)}\\
 &+\|a_0\|_{L^\infty}\|\sqrt{t}\na\D e^{t\D}\D_j u_0\|_{L^2_t(L^2)}\Bigr).
 \end{split}
 \eeno
 Whereas it follows from the law of product, \eqref{S3eq0}, that
 \beno
 \begin{split}
 \|\sqrt{t}F_j\|_{L^2_t(L^6)}\leq C\|\sqrt{t}F_j\|_{L^2_t(\dH^1)}\leq &C\bigl(\|e^{t\D}u_0\|_{L^2_t(B^{\f32})}\|\sqrt{t}\na^2 e^{t\D}\D_j u_0^\h\|_{L^\infty_t(L^2)}  \\
 &+  \|\sqrt{t}\na_\h e^{t\D}u_0^3 \|_{L^2_t(B^{\f32})}\|\na e^{t\D}\D_j u_0^\h\|_{L^\infty_t(L^2)}\bigr)  \\
 \leq &Cd_j2^{\f{j}2} \|u_0\|_{B^{\f12}} \|u_0^\h\|_{B^{\f12}}.
 \end{split}
 \eeno
 As a result, we deduce from Lemma \ref{S4lem0} and \eqref{S4eq19},  \eqref{S4eq19b} that
 \beq\label{S4eq24}
\|\sqrt{t}(\na^2v_j,\na\pi_j)\|_{L^2_t(L^6)}\leq C \sqrt{\eta} d_j2^{\f{j}2}
\eeq for $\eta$ given by \eqref{small1}.
Then we get, by a similar derivation \eqref{S3eq24}, that
\beno
\bigl\|t\bigl(\D u\cdot\na v_j+2\sum_{i=1}^3\p_iu\cdot\na\p_i v_j-\na u\cdot\na\pi_j\bigr)\bigr\|_{L^2_t(L^2)}\leq
C\sqrt{\eta} d_j2^{\f{j}2}.
\eeno

Note that
\beno
\begin{split}
\bigl\|tD_t\bigl(v_j\cdot\na e^{t\D}u_0\bigr)\bigr\|_{L^2_t(L^2)}
\leq & \| \sqrt{t}D_tv_j\|_{L^\infty_t(L^2)}\|\sqrt{t}e^{t\D}u_0 \|_{L^2_t(B^{\f52})}\\
&+\|\na v_j\|_{L^\infty_t(L^2)}\|tD_t\na e^{t\D}u_0\|_{L^2_t(B^{\f12})}.
\end{split}
\eeno
Due to Lemma \ref{S4lem4}, we have
\beq \label{S4eq26}
\begin{split}
\|tD_t\na e^{t\D}u_0\|_{L^2_t(B^{\f12})}\leq & C\bigl(\|t\p_te^{t\D}u_0\|_{L^2_t(B^{\f32})}+\|\sqrt{t}u\|_{L^\infty_t(B^{\f32})}\|\sqrt{t}e^{t\D}u_0\|_{L^2_t(B^{\f52})}\\
\lesssim &\|u_0\|_{B^{\f12}}\bigl(1+\|u_0\|_{B^{\f12}}\bigr),
\end{split}
\eeq
 Hence by virtute of Lemmas  \ref{S4lem2}, we obtain
 \beno
 \bigl\|tD_t\bigl(v_j\cdot\na e^{t\D}u_0\bigr)\bigr\|_{L^2_t(L^2)}\leq C\sqrt{\eta} d_j2^{\f{j}2}.  \eeno
Similarly, we deduce from Lemma \ref{S4lem0} that
\beno
\begin{split}
\bigl\|t\r aD_t\D e^{t\D}\D_ju_0\|_{L^2_t(L^2)}\leq &C\|a_0\|_{L^\infty}\bigl(
\bigl\|t\D\p_te^{t\D}\D_ju_0\bigr\|_{L^2_t(L^2)}\\
&\qquad\qquad\quad+\|u\|_{L^2_t(L^\infty)} \bigl\|t\na^3 e^{t\D}\D_ju_0\bigr\|_{L^\infty_t(L^2)} \bigr)\\
\leq  & Cd_j2^{\f{j}2}\|a_0\|_{L^\infty} \|u_0\|_{B^{\f12}}\bigl(1+\|u_0\|_{B^{\f12}}\bigr).
\end{split}
\eeno
Finally, we have
\beno
\begin{split}
\|tD_tF_j\|_{L^2_t(L^2)}\leq &\|tD_te^{t\D}u_0\|_{L^2_t(B^{\f32})}\bigl\|\na e^{t\D}\D_ju_0^\h\bigr\|_{L^\infty_t(L^2)}\\
&+\|e^{t\D}u_0\|_{L^2_t(B^{\f32})}\bigl\|tD_t\na e^{t\D}\D_ju_0^\h\bigr\|_{L^\infty_t(L^2)} \\
&+\|\sqrt{t}\na_\h e^{t\D}u_0^3\|_{L^2_t(B^{\f32})}\bigl\|\sqrt{t}D_t e^{t\D}\D_ju_0^\h\bigr\|_{L^\infty_t(L^2)}
\\
&+\|{t}D_t\na_\h e^{t\D}u_0^3\|_{L^2_t(B^{\f12})}\bigl\|\na e^{t\D}\D_ju_0^\h\bigr\|_{L^\infty_t(L^2)}.
\end{split}
\eeno
Whereas it follows from Lemma \ref{S4lem4} that
\beno
\begin{split}
\|tD_te^{t\D}u_0\|_{L^2_t(B^{\f32})}\leq &\|t\p_te^{t\D}u_0\|_{L^2_t(B^{\f32})}+\|\sqrt{t}u\|_{L^\infty_t(B^{\f32})}\|\sqrt{t}e^{t\D}u_0\|_{L^2_t(B^{\f52})}\\
\leq &C\|u_0\|_{B^{\f12}}\bigl(1+\|u_0\|_{B^{\f12}}\bigr).
\end{split}
\eeno
And it follows from Lemma \ref{S4lem0} that
\beno
\begin{split}
\bigl\|tD_t\na e^{t\D}\D_ju_0^\h\bigr\|_{L^\infty_t(L^2)}\leq &\bigl\|t\p_t\na e^{t\D}\D_ju_0^\h\bigr\|_{L^\infty_t(L^2)}
+\|\sqrt{t}u\|_{L^\infty_t(B^{\f32})}\bigl\|\sqrt{t}\na^2 e^{t\D}\D_ju_0^\h\bigr\|_{L^\infty_t(L^2)}\\
\leq &Cd_j2^{\f{j}2} \bigl(1+\|u_0\|_{B^{\f12}}\bigr) \|u_0^\h\|_{B^{\f12}},
\end{split}
\eeno
and
\beno
\begin{split}
\bigl\|\sqrt{t}D_t e^{t\D}\D_ju_0^\h\bigr\|_{L^\infty_t(L^2)}\leq &\bigl\|\sqrt{t}\p_t e^{t\D}\D_ju_0^\h\bigr\|_{L^\infty_t(L^2)}
+\|\sqrt{t}u\|_{L^\infty_t(B^{\f32})}\bigl\|\na e^{t\D}\D_ju_0^\h\bigr\|_{L^\infty_t(L^2)}\\
\leq &Cd_j2^{\f{j}2} \bigl(1+\|u_0\|_{B^{\f12}}\bigr) \|u_0^\h\|_{B^{\f12}}.
\end{split}
\eeno
Hence we obtain
\beno
\|tD_tF_j\|_{L^2_t(L^2)}
\leq  C d_j2^{\f{j}2}\bigl(1+\|u_0\|_{B^{\f12}}\bigr) \|u_0\|_{B^{\f12}} \|u_0^\h\|_{B^{\f12}}.\eeno

Therefore, by summing up the above estimates, we achieve
\beq\label{S4eq25}
\|tG_j\|_{L^2_t(L^2)}\leq C \sqrt{\eta}  d_j2^{\f{j}2}.
\eeq
With the estimate \eqref{S4eq25}, we can follow the proof of Proposition \ref{S3prop3} to prove that
\beq \label{S4eq30}
\|t\na D_tv_j\|_{L^\infty_t(L^2)}^2+\bigl\|t(D_t^2v_j, \na^2D_tv_j, \na D_t\pi_j)\bigr\|_{L^2_t(L^2)}^2\leq  C \eta  d_j^22^{{j}}.
\eeq

Thanks to \eqref{S4eq20} and \eqref{S4eq30}, we deduce  \eqref{S4eq21} via a similar derivation of
\eqref{S3eq7}. This completes the proof of the proposition.
\end{proof}

\medskip

 \setcounter{equation}{0}
  \section{The proof of Theorem \ref{thm3}}\label{Sect5}

In this section, we shall modify the proof of Theorem \ref{thm1} to prove Theorem \ref{thm3}.  Let
us first present the proof of Proposition \ref{S5prop1}.

\begin{proof}[Proof of Proposition \ref{S5prop1}] Let $(\r,u)$ be a smooth enough solution of
 \eqref{S1eq1} on $[0,T^\ast[.$ We  construct $(u_j,\na\pi_j)$ via \eqref{S3eq1}. Then there holds \eqref{S3eq1a}.
With $u_0\in\dH^{\f12+2\ga},$ we  deduce from \eqref{S3eq2} that
\beq \label{S5eq2}
\|u_j\|_{L^\infty(L^2)}^2+\|\na u_j\|_{L^2_t(L^2)}^2\leq C\|\D_j u_0\|_{L^2}^2\leq Cc_j^22^{-2j\left(\f12+2\ga\right)}\|u_0\|_{\dH^{\f12+2\ga}}^2.
\eeq
Here and in the rest of this section, we always denote $(c_j)_{j\in\Z}$ to be a generic element of $\ell^2(\Z)$ so that
$\sum_{j\in\Z}c_j^2=1.$

Whereas thanks to \eqref{S3eq2a} and \eqref{S3eq3}, we infer that there exists a positive constant $c$ so that
\beno \f{d}{dt}\|\na u_j(t)\|_{L^2}^2+2c\bigl\|(\p_tu_j,\na^2 u_j,\na\pi_j)\bigr\|_{L^2}^2\leq C\|u\cdot\na u_j\|_{L^2}^2\quad\mbox{for}\ t<T^\ast.
\eeno
Yet due to $\ga\in ]0,1/4],$  it follows from the law of product in homogeneous Sobolev spaces, \eqref{S3eq0}, that
\beq\label{S5eq16}
\begin{split}
\|u\cdot\na u_j\|_{L^2}\leq &C\|u\|_{\dot H^{\f12+2\ga}}\|\na u_j\|_{\dot H^{1-2\ga}}\\
\leq& C\|u\|_{\dot H^{\f12+2\ga}}\|\na u_j\|_{L^2}^{2\ga}\|\na^2 u_j\|_{L^2}^{1-2\ga}.
\end{split}
\eeq
Then applying Young's inequality gives
\beno
C\|u\cdot\na u_j\|_{L^2}^2
\leq \f{C}\ga\|u\|_{\dot H^{\f12+2\ga}}^{\f1\ga}\|\na u_j\|_{L^2}^{2}+{c}\|\na^2 u_j\|_{L^2}^{2}.
\eeno
Thus we achieve
\beq \label{S5eq7} \f{d}{dt}\|\na u_j(t)\|_{L^2}^2+c\bigl\|(\p_tu_j,\na^2 u_j,\na\pi_j)\bigr\|_{L^2}^2\leq \f{C}\ga\|u\|_{\dot H^{\f12+2\ga}}^{\f1\ga}\|\na u_j\|_{L^2}^{2}
\quad\mbox{for}\ t<T^\ast.
\eeq
Applying Gronwall's inequality leads to
\beq \label{S5eq1}
 \begin{split}
 \|\na u_j\|_{L^\infty_t(L^2)}^2+&c\bigl\|(\p_tu_j,\na^2 u_j,\na\pi_j)\bigr\|_{L^2_t(L^2)}^2\\
 \leq &C\|\na\D_j u_0\|_{L^2}^2\exp\Bigl(\f{C}\ga\int_0^t\|u\|_{\dot H^{\f12+2\ga}}^{\f1\ga}\,dt'\Bigr)\\
 \leq &Cc_j^22^{2j(\f12-2\gamma)}\|u_0\|_{\dH^{\f12+2\ga}}^2\exp\Bigl(\f{C}\ga\int_0^t\|u\|_{\dot H^{\f12+2\ga}}^{\f1\ga}\,dt'\Bigr)\quad\mbox{for}\ t<T^\ast.
 \end{split}
 \eeq
 In view of \eqref{S5eq2} and \eqref{S5eq1}, for $t<t^\ast,$ we deduce by a similar derivation of \eqref{S3eq7} that
 \beq\label{S5eq18}
    \begin{split}
    \|\D_ju\|_{L^\infty_t(L^2)}+\|\na\D_ju\|_{L^2_t(L^2)}\lesssim & \sum_{j'>j}\bigl(\|u_j\|_{L^\infty_t(L^2)}+\|\na u_j\|_{L^2_t(L^2)}\bigr)\\
    &+2^{-j}\sum_{j'\leq j}\bigl(\|\na u_j\|_{L^\infty_t(L^2)}+\|\na^2 u_j\|_{L^2_t(L^2)}\bigr)\\
    \lesssim & c_j2^{-j\left(\f12+2\ga\right)}\|u_0\|_{\dH^{\f12+2\ga}}\exp\Bigl(\f{C}\ga\int_0^t\|u\|_{\dot H^{\f12+2\ga}}^{\f1\ga}\,dt'\Bigr),
    \end{split}
    \eeq
which implies  for $t<t^\ast$
    \beq \label{S5eq3}
    \|u\|_{\wt{L}^\infty_t(\dH^{\f12+2\ga})}^2+\|\na u\|_{L^2_t(\dH^{\f12+2\ga})}^2\leq                                                                         C \|u_0\|_{\dH^{\f12+2\ga}}^2\exp\Bigl(\f{C}\ga\int_0^t\|u\|_{\dot H^{\f12+2\ga}}^{\f1\ga}\,dt'\Bigr).
    \eeq
Let us denote
\beq \label{S5eq4}
T^\star_4\eqdefa \bigl\{\  t< T^\ast\     \  \|u\|_{\wt{L}^\infty_t(\dH^{\f12+2\ga})}\leq 2C  \|u_0\|_{\dH^{\f12+2\ga}}  \ \bigr\}.
\eeq
Then for $t\leq \min\Bigl(T^\star_4, \f{\gamma \ln\bigl(\f{3}{C}\bigr)}{2^{\f1{\gamma}}C^{1+\f1\ga}\|u_0\|_{\dot H^{\f12+2\ga}}^{\f1\ga} }\Bigr),$
we deduce from \eqref{S5eq3} that
\beq \label{S5eq5} \|u\|_{\wt{L}^\infty_t(\dH^{\f12+2\ga})}^2+\|\na u\|_{L^2_t(\dH^{\f32+2\ga})}^2\leq
 {3C} \|u_0\|_{\dH^{\f12+2\ga}}^2.
\eeq
This in turn shows that
\beno
T^\star\geq \f{\gamma \ln\bigl(\f{3}{C}\bigr)}{2^{\f1{\gamma}}C^{1+\f1\ga}\|u_0\|_{\dot H^{\f12+2\ga}}^{\f1\ga} }\eqdefa T_\ga,
\eeno
and for $t\leq T_\ga,$ there holds \eqref{S5eq11}. This completes the proof of Proposition \ref{S5prop1}.
\end{proof}

\begin{rmk}
It follows from the definition of $T_\ga$ that
\beq \label{S5eq13}
\exp\Bigl(\f{C}\ga\int_0^{T_\ga}\|u(t)\|_{\dot H^{\f12+2\ga}}^{\f1\ga}\,dt\Bigr)\leq 4.
\eeq
\end{rmk}

\begin{proof}[Proof of Proposition \ref{S5prop2}]
Due to the energy conservation law, we have
 \beno
  \f12\|\sqrt{\r}u\|_{L^\infty_t(L^2)}^2+\|\na u\|_{L^2_t(L^2)}^2\leq \f12 \|\sqrt{\r_0}u_0\|_{L^2}^2 \quad\mbox{for}\ \ t<T^\ast,
 \eeno
 from which and \eqref{S5eq11}, we obtain
 \beq \label{S5eq6}
 \|u\|_{L^2_t(B^{\f32})}\leq C\|\na u\|_{L^2_t(L^2)}^{\f{4\ga}{1+4\ga}}   \|\na u\|_{L^2_t(\dH^{\f12+2\ga})}^{\f{1}{1+4\ga}}\leq C\|u_0\|_{H^{\f12+2\ga}}\quad\mbox{for}\ \ t\leq T_\ga.
 \eeq
On the other hand,  by multiplying \eqref{S5eq7} by $t$ and then applying Gronwall's inequality, we get
\beno
\begin{split}
\|\sqrt{t}\na u_j\|_{L^\infty_t(L^2)}^2+&c\bigl\|\sqrt{t}(\p_tu_j, \na^2 u_j, \na\pi_j)\bigr\|_{L^2_t(L^2)}^2\\
\leq &C\|\na u_j\|_{L^2_t(L^2)}^2\exp\Bigl(\f{C}\ga\int_0^t\|u\|_{\dot H^{\f12+2\ga}}^{\f1\ga}\,dt'\Bigr).
\end{split}
\eeno
which together with \eqref{S5eq2} and \eqref{S5eq13} implies that for $t\leq T_\ga$
\beq \label{S5eq8}
  \|\sqrt{t}\na u_j\|_{L^\infty_t(L^2)}^2+c\bigl\|\sqrt{t}(\p_tu_j, \na^2 u_j, \na\pi_j)\bigr\|_{L^2_t(L^2)}^2
  \leq Cc_j^22^{-2j\left(\f12+2\gamma\right)}\|u_0\|_{\dH^{\f12+2\ga}}^2. \eeq

 With \eqref{S5eq6}, we get, by a similar derivation of \eqref{S5eq8}, that
 \beno
 \|\p_tu_j\|_{L^2_t(L^2)}\lesssim d_j 2^{\f{j}2}\|u_0\|_{B^{\f12}}\andf
         \|\sqrt{t}\p_tu_j\|_{L^2_t(L^2)}\lesssim d_j 2^{-\f{j}2}\|u_0\|_{B^{\f12}},
         \eeno which implies
 \beq \label{S5eq9}
 \|t^{\f14}u_t\|_{L^2_t(L^2)}\leq C\|u_0\|_{B^{\f12}}\leq C\|u_0\|_{H^{\f12+2\ga}}.
 \eeq
   Then by virtue of \eqref{S3eq17a}, \eqref{S5eq6} and \eqref{S5eq9}, we achieve
  \beq \label{S5eq15}
  \begin{split}
  \|\sqrt{t}\p_tu_j\|_{L^\infty_t(L^2)}^2+\|\sqrt{t}\na\p_tu_j\|_{L^2_t(L^2)}^2\leq& C\exp\Bigl(\int_0^t\bigl(\|u\|_{B^{\f32}}^2+\|t^{\f14}u_t\|_{L^2}^2\bigr)\,dt'\Bigr)
  \\
  \times \bigl(\|&\sqrt{\r}\p_tu_j\|_{L^2_t(L^2)}^2
  +\|t^{\f14}u_t\|_{L^2_t(L^2)}^2\|\na u_j\|_{L^\infty_t(L^2)}^2\bigr)\\
  \leq &   Cc_j^22^{2j\left(\f12-2\gamma\right)}\|u_0\|_{\dH^{\f12+2\ga}}^2.
  \end{split}
  \eeq
Whereas we deduce from \eqref{S3eq3} and \eqref{S5eq16} that
\beno
\begin{split}
\|\na^2u_j\|_{L^2}\leq &C\bigl(\|\p_tu_j\|_{L^2}+\|u\cdot\na u_j\|_{L^2}\bigr)\\
\leq &C\bigl(\|\p_tu_j\|_{L^2}+\|u\|_{\dot H^{\f12+2\ga}}\|\na u_j\|_{L^2}^{2\ga}\|\na^2 u_j\|_{L^2}^{1-2\ga}\bigr)\\
\leq &C\bigl(\|\p_tu_j\|_{L^2}+\|u\|_{\dot H^{\f12+2\ga}}^{\f1{2\ga}}\|\na u_j\|_{L^2}\bigr)+\f12\|\na^2 u_j\|_{L^2}.
\end{split}
\eeno
which together with \eqref{S5eq1} and \eqref{S5eq15} ensures that
\beq\label{S5eq17}
\begin{split}
\|\sqrt{t}\na^2u_j\|_{L^\infty_t(L^2)}\leq & C\bigl(\|\sqrt{t}\p_tu_j\|_{L^\infty_t(L^2)}+\sqrt{T_\ga}\|u\|_{L^\infty_t(\dot H^{\f12+2\ga})}^{\f1{2\ga}}\|\na u_j\|_{L^\infty_t(L^2)}\bigr)\\
 \leq &   Cc_j2^{j\left(\f12-2\gamma\right)}\|u_0\|_{\dH^{\f12+2\ga}}.
  \end{split}
  \eeq
Thanks to \eqref{S5eq8}, \eqref{S5eq15} and \eqref{S5eq17},  we conclude the proof of \eqref{S5eq10} by following the same line
as \eqref{S5eq18}.
\end{proof}

\begin{rmk}\label{S5rmk2}
Thanks to \eqref{S5eq6} and \eqref{S5eq9}, we get, by a similar proof of Corollary \ref{S3col2}, that
\beq \label{S5eq19}
\begin{split}
\|t D_tu_j\|_{L^\infty_t(L^2)}+\|t \na D_tu_j\|_{L^2_t(L^2)}
\leq  Cc_j2^{-j\left(\f{1}2+2\ga\right)}\|u_0\|_{\dH^{\f12+2\ga}},
\end{split}
\eeq
and there holds \eqref{S3eq16} for $t\leq T_\ga.$

Furthermore, by virtue of \eqref{S5eq15} and \eqref{S5eq17}, we get, by a similar derivation of \eqref{S3eq22}, that
\beq
\label{S5eq20}
\bigl\|\sqrt{t}(\na^2u_j,\na\pi_j)\bigr\|_{L^2_t(L^6)}\leq Cc_j2^{j\left(\f{1}2-2\ga\right)}\|u_0\|_{\dH^{\f12+2\ga}}
\quad\mbox{for} \ \ t\leq T_\ga.
\eeq
\end{rmk}

With Remark \ref{S5rmk2}, we can follow the proof of Proposition \ref{S3prop3} to conclude the proof of  Proposition \ref{S5prop3}, which
we omit the details here.

\medskip

\appendix

\setcounter{equation}{0}
\section{Tool box on Littlewood-Paley theory}\label{apA}

For the convenience of the readers, we  recall some basic facts on  Littlewood-Paley theory from \cite{BCD}. Let
\begin{equation}\label{S1eq5}
\Delta_j a\eqdefa \cF^{-1}(\varphi(2^{-j}|\xi|)\widehat{a}),
 \andf  S_j a\eqdefa \cF^{-1}(\chi(2^{-j}|\xi|)\widehat{a}),
\end{equation}
where $\cF a$ and
$\widehat{a}$ denote the Fourier transform of the distribution $a,$
$\chi(\tau)  $ and~$\varphi(\tau)$ are smooth functions such that
\beno
&&\Supp \varphi \subset \Bigl\{\tau \in \R\,/\  \ \frac34 \leq
\tau| \leq \frac83 \Bigr\}\ \andf \  \ \forall
\tau>0\,,\ \sum_{j\in\Z}\varphi(2^{-j}\tau)=1,\\
&&\Supp \chi \subset \Bigl\{\tau \in \R\,/\  \ \ |\tau|  \leq
\frac43 \Bigr\}\quad \ \ \ \andf \  \ \forall
\tau\in\R\,,\  \chi(\tau)+ \sum_{j\geq 0}
\varphi(2^{-j}\tau)=1.
\eeno

\begin{defi}\label{defbesov}
{\sl  Let $(p, r)$ be in~$[1,\infty]^2$ and~$s$ in~$\R$. Let us consider~$u$ in~$\cS_h'(\R^3),$ which means that $u$ is in~$\cS'(\R^3)$ and satisfies~$\ds\lim_{j\to-\infty}\|S_ju\|_{L^\infty}=0$. We set
$$
\|u\|_{\dB^s_{p,r}}\eqdefa\big\|\big(2^{js}\|\Delta_j
u\|_{L^{p}}\big)_{j\in\Z}\bigr\|_{\ell ^{r}(\ZZ)}.
$$
\begin{itemize}
\item
if $s<\frac{3}{p}$ (or $s=\frac{3}{p}$ if $r=1$), we define $
\dB^s_{p,r}(\R^3)\eqdefa \big\{u\in{\mathcal S}_h'(\R^3)\;\big|\; \|u\|_{\dB^s_{p,r}}<\infty\big\}.$
\item
 if $k\in\N$ and if~$\frac{3}{p}+k\leq
\frac{3}{p}+k+1$ (or $s=\frac{3}{p}+k+1$ if $r=1$), then we
fine~$ \dB^s_{p,r}(\R^3)$  as the subset of $u$
~${\mathcal S}_h'(\R^3)$ such that $\partial^\beta u$ belongs to~$
\dB^{s-k}_{p,r}(\R^3)$ whenever $|\beta|=k.$
\end{itemize}}
\end{defi}
We remark that $\dB^s_{2,2}$
coincides with the classical homogeneous Sobolev spaces $\dH^s$.

In  order to obtain a better description of the regularizing effect
of the transport-diffusion equation, we will use Chemin-Lerner type
spaces $\widetilde{L}^{q}_T(B^s_{p,r}(\R^3))$.
\begin{defi}\label{chaleur+}
Let $s\leq\frac{3}{p}$ (respectively $s\in\R$),
$(q,p,r)\in[1,\,+\infty]^3$ and $T\in]0,\,+\infty]$. We define
$\widetilde{L}^{q}_T(B^s_{p\,r}(\R^3))$ as the completion of
$C([0,T];\cS_h(\R^3))$ by the norm
$$
\| f\|_{\widetilde{L}^{q}_T(B^s_{p,r})} \eqdefa
\Big(\sum_{j\in\Z}2^{jrs} \Big(\int_0^T\|\Delta_j\,f(t)
\|_{L^p}^{q}\, dt\Big)^{\frac{r}{q}}\Big)^{\frac{1}{r}}
<\infty,
$$
with the usual change if $r=\infty.$  For short, we just denote this
space by $\widetilde{L}^{q}_T(B^s_{p,r}).$
\end{defi}

We also frequently use the following Lemmas from \cite{BCD}:

\begin{lem}[Corollary 5.5 of \cite{BCD}]\label{S3lem0}
{\sl Let $(s_1,s_2)\in \left]-3/2, 3/2\right[^2,$ a constant $C$ exists such that if $s_1+s_2$ is positive, then we have
 \beq \label{S3eq0}
 \|fg\|_{\dB^{s_1+s_2-\f32}_{2,1}}\leq C\|f\|_{\dH^{s_1}}\|g\|_{\dH^{s_2}}.
\eeq
}
\end{lem}

\begin{lem}[Lemma 2.4 of \cite{BCD}]\label{S4lem0}
{\sl Let $\cC$ be an annulus. Positive constants $c$ and $C$ exist such that for any
$p\in [1,\infty]$ and any couple $(t,\la)$ of positive real numbers, we have
\beno
{\rm Supp}\hat{u} \subset \la\cC \Rightarrow \|e^{t\D}f\|_{L^p}\leq C e^{-ct\la^2}\|f\|_{L^p}.
\eeno}
\end{lem}

\bigbreak \noindent {\bf Acknowledgments.} The author would like to thank Professor Jean-Yves Chemin for introducing him
the reference \cite{CG}.  P. Zhang is partially supported
    by NSF of China under Grants   11371347 and 11688101,  and innovation grant from National Center for
    Mathematics and Interdisciplinary Sciences of The Chinese Academy of Sciences.

\end{document}